\documentclass[11pt,a4paper,oldfontcommands]{article}
\usepackage[utf8]{inputenc}
\usepackage[T1]{fontenc}
\usepackage{microtype}
\usepackage[dvips]{graphicx}
\usepackage{xcolor}
\usepackage{times}

\usepackage[english]{babel}
\usepackage{amsmath,mathtools}
\usepackage{amsfonts}
\usepackage{mathrsfs}
\usepackage{amsthm}
\usepackage{subfigure}
\usepackage{envmath}
\usepackage{amssymb}
\usepackage{dsfont}
\usepackage{bigints}
\usepackage{textgreek}
\usepackage{authblk}
\usepackage{titlesec}
\usepackage{enumitem}
\usepackage{url}
\usepackage{tkz-fct}
\usepackage{tikz}
\usetikzlibrary{patterns}
\usepackage{verbatim}

\usepackage[
breaklinks=true,
colorlinks=false,
bookmarks=true,bookmarksopenlevel=2]{hyperref}

\usepackage{geometry}
\newtheorem{theorem}{Theorem}[section]
\newtheorem{prop}[theorem]{Proposition}
\newtheorem{lem}[theorem]{Lemma}
\newtheorem{cor}[theorem]{Corollary}

\theoremstyle{remark}
\newtheorem{rem}[theorem]{Remark}

\numberwithin{equation}{section}

\renewcommand{\Re}{\operatorname{Re}}

\renewcommand{\Im}{\operatorname{Im}}

\newcommand{\Tr}{\operatorname{Tr}}

\newcommand*\diff{\mathop{}\!\mathrm{d}}

\DeclareMathOperator\supp{supp}

\DeclarePairedDelimiter\abs{\lvert}{\rvert}%
\DeclarePairedDelimiter\norm{\lVert}{\rVert}%

\makeatletter
\let\oldabs\abs
\def\abs{\@ifstar{\oldabs}{\oldabs*}}
\let\oldnorm\norm
\def\norm{\@ifstar{\oldnorm}{\oldnorm*}}
\makeatother

\newcommand*{\myemail}[1]{%
    \normalsize\href{mailto:#1}{#1}\par
    }

\allowdisplaybreaks

\titleformat{\section}[block]{\centering \scshape \large}{\thesection.}{0.3\baselineskip}{}
\titlespacing{\section}{0pt}{*5}{*2}

\titleformat{\subsection}[block]{\bfseries}{\thesubsection.}{.5em}{}
\titlespacing{\subsection}{0pt}{*2.5}{*1}

\titleformat{\subsubsection}[runin]{\itshape}{\normalfont \thesubsubsection.}{.5em}{}[.]
\titlespacing{\subsubsection}{0pt}{*2.5}{0.5em}

\titleformat{\section}[block]{\centering \scshape \large}{\thesection.}{0.3\baselineskip}{}
\titlespacing{\section}{0pt}{*5}{*2}

\titleformat{\subsection}[block]{\bfseries}{\thesubsection.}{.5em}{}
\titlespacing{\subsection}{0pt}{*2.5}{*1}

\titleformat{\subsubsection}[runin]{\itshape}{\normalfont \thesubsubsection.}{.5em}{}[.]
\titlespacing{\subsubsection}{0pt}{*2.5}{0.5em}

\title{Existence of multi-solitons for the focusing Logarithmic Non-Linear Schrödinger Equation}
\date{\vspace{-1cm}}

\author[]{Guillaume Ferriere}

\affil[]{IMAG, Univ Montpellier, CNRS, Montpellier, France \\ \myemail{guillaume.ferriere@umontpellier.fr}}

\begin{document}

\maketitle

\begin{abstract}
    We consider the logarithmic Schrödinger equation (logNLS) in the focusing regime. For this equation, Gaussian initial data remains Gaussian. In particular, the Gausson - a time-independent Gaussian function - is an orbitally stable solution. In this paper, we construct \emph{multi-solitons} (or \emph{multi-Gaussons}) for logNLS, with estimates in $H^1 \cap \mathcal{F}(H^1)$. We also construct solutions to logNLS behaving (in $L^2$) like a sum of $N$ Gaussian solutions with different speeds (which we call \emph{multi-gaussian}). In both cases, the convergence (as $t \rightarrow \infty$) is faster than exponential. We also prove a rigidity result on these multi-gaussians and multi-solitons, showing that they are the only ones with such a convergence.
\end{abstract}

\section{Introduction}

\subsection{Setting}

We are interested in the \textit{Logarithmic Non-Linear Schrödinger Equation}
\begin{equation}
    i \, \partial_t u + \frac{1}{2} \Delta u + \lambda u \ln{\abs{u}^2} = 0,
    \label{foc_log_nls}
\end{equation}
with $x \in \mathbb{R}^d$, $d \geq 1$, $\lambda \in \mathbb{R} \setminus \{ 0 \}$. This equation was introduced as a model of nonlinear wave mechanics and in nonlinear optics (\cite{nonlin_wave_mec}, see also \cite{inco_white_light_log, log_nls_nuclear_physics, quantal_damped_motion, solitons_log_med, log_nls_magma_transp}). The case $\lambda < 0$ (whose study of the Cauchy problem goes back to \cite{cazenave-haraux, Guerrero_Lopez_Nieto_H1_solv_lognls}) was recently studied by R. Carles and I. Gallagher who made explicit an unusually faster dispersion with a universal behaviour of the modulus of the solution (see \cite{carlesgallagher}). The knowledge of this behaviour was very recently improved with a convergence rate but also extended through the semiclassical limit in \cite{Ferriere__Wass_semiclass_defoc_NLS}. On the other hand, the case $\lambda > 0$ seems to be the more interesting from a physical point of view and has been studied formally and rigorously (see for instance \cite{Dav_Mont_Squa_lognls, quantal_damped_motion}). In particular, the existence and uniqueness of solutions to the Cauchy problem have been solved in \cite{cazenave-haraux}.
Moreover, it has been proved to be the non dispersive case and also that the so called \emph{Gausson}
\begin{equation}
    G^d (x) \coloneqq \exp \Bigl( \frac{d}{2} - \lambda \abs{x}^2 \Bigr), \qquad x \in \mathbb{R}^d, \label{expr_gausson}
\end{equation}
and its derivates through the invariants of the equation (translation in space, Galilean invariance, multiplication by a complex constant) are explicit solutions to \eqref{foc_log_nls} and bound states for the energy functional. Several results address the orbital stability of the Gausson as well as the existence of other stationary solutions and Gaussian solutions to \eqref{foc_log_nls}; see \textit{e.g.} \cite{nonlin_wave_mec, Cazenave_log_nls, Dav_Mont_Squa_lognls, Ardila__Orbital_stability_Gausson}.
In this article, we address the question of the existence of multi-solitons (\textit{i.e.} multi-Gaussons), but also the existence of multi-gaussians.

\begin{rem}[Effect of scaling factors] \label{rem:scaling}
    As noticed in \cite{carlesgallagher}, unlike what happens in the case of an homogeneous nonlinearity (classically of the form $\abs{u}^p u$), replacing $u$ with $\kappa u$ ($\kappa > 0$) in \eqref{foc_log_nls} has only little effect, since we have
    \begin{equation*}
        i \, \partial_t (\kappa u) + \frac{1}{2} \Delta (\kappa u) + \lambda (\kappa u) \ln{\lvert \kappa u \rvert^2} - 2 \lambda (\ln{\kappa}) \kappa u = 0.
    \end{equation*}
    The scaling factor thus corresponds to a purely time-dependent gauge transform:
    \begin{equation*}
        \kappa u(t,x) \, e^{-2it \lambda \ln{\kappa}}
    \end{equation*}
    solves \eqref{foc_log_nls}. In particular, the $L^2$-norm of the initial datum does not influence the dynamics of the solution.
\end{rem}

\subsection{The Logarithmic Non-Linear Schrödinger Equation}

The Logarithmic Non-Linear Schrödinger Equation was introduced by I. Bia{\l}ynicki-Birula and J. Mycielski (\cite{nonlin_wave_mec}) who proved that it is the only nonlinear Schrödinger theory in which \emph{the separability of noninteracting systems} hold: for noninteracting subsystems, no correlations are introduced by the nonlinear term. Therefore, for any initial data of the form $u_\textnormal{in} = u_\textnormal{in}^1 \otimes u_\textnormal{in}^2$, \textit{i.e.}
\begin{equation*}
    u_\textnormal{in} (x) = u_\textnormal{in}^1 (x_1) \, u_\textnormal{in}^2 (x_2), \qquad \forall x_1 \in \mathbb{R}^{d_1}, \forall x_2 \in \mathbb{R}^{d_2}, x = (x_1, x_2),
\end{equation*}
the solution $u$ to \eqref{foc_log_nls} (in dimension $d = d_1 + d_2$) with initial data ${{u}_|}_{t=0} = u_{\textnormal{in}}$ is
\begin{equation*}
    u (t) = u^1 (t) \otimes u^2 (t)
\end{equation*}
where $u^j$ is the solution to \eqref{foc_log_nls} in dimension $d_j$ with initial data $u_\textnormal{in}^j$ ($j = 1, 2$).

They also emphasized that the case $\lambda > 0$ is probably the most physically relevant.
For this case, the Cauchy problem has already been studied in \cite{cazenave-haraux} (see also \cite{Cazenave_semlin_lognls}). We define the energy space
\begin{equation*}
    W (\mathbb{R}^d) \coloneqq \{ v \in H^1 (\mathbb{R}^d), \abs{v}^2 \ln{\abs{v}^2} \in L^1 (\mathbb{R}^d) \},
\end{equation*}
which is a reflexive Banach space when endowed with a Luxembourg type norm (see \cite{Cazenave_log_nls}).
We can also define the mass, the angular momentum and the energy for all $v \in W (\mathbb{R}^d)$:
\begin{gather*}
    M (v) \coloneqq \norm{v}_{L^2}^2, \qquad
    \mathcal{J} (v) \coloneqq \Im \int_{\mathbb{R}^d} \overline{v} \, \nabla v \diff x, \qquad
    E(v) \coloneqq \frac{1}{2} \norm{\nabla v}_{L^2}^2 - \lambda \int_{\mathbb{R}^d} \abs{v}^2 (\ln{\abs{v}^2} - 1) \diff x.
\end{gather*}

\begin{theorem}[{\cite[Théorème~2.1]{cazenave-haraux}}, see also {\cite[Theorem~9.3.4]{Cazenave_semlin_lognls}}] \label{th:Cauchy_problem}
    For $\lambda > 0$, for any initial data $u_\textnormal{in} \in W (\mathbb{R}^d)$, there exists a unique, global solution $u \in \mathcal{C}_b (\mathbb{R}, W(\mathbb{R}^d))$. Moreover the mass $M(u(t))$, the angular momentum $\mathcal{J} (u(t))$ and the energy $E(u(t))$ are independent of time.
\end{theorem}

It is also worth noticing that there is an energy estimate at the level $L^2$:

\begin{lem}[{\cite[Lemme~2.2.1]{cazenave-haraux}}] \label{lem:L2_energy_est}
    For $\lambda > 0$, for any solutions $u$ and $v$ to \eqref{foc_log_nls} given by Theorem \ref{th:Cauchy_problem} with initial data $u_\textnormal{in}, v_\textnormal{in} \in W (\mathbb{R}^d)$ respectively, there holds for all $t \in \mathbb{R}$
    \begin{equation*}
        \norm{u(t) - v(t)}_{L^2} \leq e^{2 \lambda \abs{t}} \, \norm{u_\textnormal{in} - v_\textnormal{in}}_{L^2}.
    \end{equation*}
\end{lem}

Another surprising feature of \eqref{foc_log_nls} is that any Gaussian data remains Gaussian (\cite{nonlin_wave_mec}).

\begin{prop} \label{prop:expression_general_gaussian}
Any Gaussian initial data
\begin{equation*}
    \exp \Bigl[ \frac{d}{2} - x^\top A^\textnormal{in} x \Bigr], \label{gauss_in_data}
\end{equation*}
with $A^\textnormal{in} \in S_d (\mathbb{C})^{\Re +} \coloneqq \{ M \in M_d (\mathbb{C}), M^\top = M, \, \Re A^\textnormal{in} \in S_d (\mathbb{R})^{++} \}$ (where $^\top$ designates the transposition), gives rise to a Gaussian solution $G^{A^\textnormal{in}}$ to \eqref{foc_log_nls} of the form
\begin{equation}
    B^{A^\textnormal{in}} (t,x) \coloneqq \biggl( \frac{\det{\Re A(t)}}{\det{\Re A^\textnormal{in}}} \biggr)^\frac{1}{4} \exp \Bigl[ \frac{d}{2} - i \, \Phi (t) - \frac{1}{2} x^\top A (t) x \Bigr], \label{eq:gen_gauss_sol}
\end{equation}
where $A$ and $\phi$ satisfy
\begin{gather}
    \frac{\diff A}{\diff t} = - i A(t)^2 + 2 i \lambda \Re A(t), \qquad \qquad A(0) = A^\textnormal{in}, \label{eq:matrix_ODE} \\
    \Phi (t) \coloneqq \frac{1}{2} \int_0^t \Tr ( \Re A (s) ) \diff s - \frac{\lambda}{2} \int_0^t \ln \biggl( \frac{\det \Re A(s)}{\det \Re A^\textnormal{in}} \biggr) \diff s - d \lambda t. \notag
\end{gather}
Moreover, if $\lambda > 0$, $$0 < \inf_t \sigma ( \Re A (t) ) \leq \sup_t \sigma ( \Re A (t) ) < + \infty.$$
\end{prop}

In parallel, we define their derivates through the invariants and the scaling effect:
\begin{equation}
    B^{A^\textnormal{in}}_{\omega, x_0, v, \theta} (t,x) \coloneqq \exp \left[ i \left( \theta + 2 \lambda \omega t - v \cdot x + \frac{\abs{v}^2}{2} t \right) + \omega \right] B^{A^\textnormal{in}} (t, x - x_0 - v t), \qquad t \in \mathbb{R}, \, x \in \mathbb{R}^d, \label{eq:exp_der_gen_gauss}
\end{equation}

For such data, the evolution of the solution is given by a single matrix ODE, which can even be simplified in dimension 1 (see \cite{carlesgallagher,Bao_Carles_al__error_est,Ferriere__superposition_logNLS}):

\begin{prop} \label{prop_existence_breathers_d=1}
    For any $\alpha \in \mathbb{C}^+ \coloneqq \{ z \in \mathbb{C}, \Re z > 0 \}$, consider the ordinary differential equation
    \begin{equation*}
        \Ddot{r}_\alpha = \frac{1}{r_\alpha^3} - \frac{2 \lambda}{r_\alpha}, \qquad r_\alpha (0) = \Re \alpha \eqqcolon \alpha_r, \qquad \dot{r}_\alpha (0) = \Im \alpha \eqqcolon \alpha_i.
    \end{equation*}
    It has a unique solution $r_{\alpha} \in \mathcal{C}^\infty (\mathbb{R})$ with values in $(0, \infty)$.
    Then, set
    \begin{equation}
        u^{\alpha} (t,x) \coloneqq \sqrt{\frac{\alpha_r}{r_\alpha (t)}} \exp \Bigl[ \frac{1}{2} - i \Phi (t) - \frac{x^2}{2 r_\alpha (t)^2} + i \frac{\dot{r}_\alpha (t)}{r_\alpha (t)} \frac{x^2}{2} \Bigr], \qquad t,x \in \mathbb{R}, \label{eq:exp_breather_1d}
    \end{equation}
    where
    \begin{equation*}
        \Phi (t) \coloneqq \frac{1}{2} \int_0^t \frac{1}{r_\alpha (s)^2} \diff s + \lambda \int_0^t \ln \frac{r_\alpha (s)}{\alpha_r} \diff s - \lambda t.
    \end{equation*}
    Then $u^{\alpha}$ is solution to \eqref{foc_log_nls} in dimension $d = 1$.
\end{prop}

Note that whichever the sign of $\lambda$, the energy $E$ has no definite sign. The distinction between focusing or defocusing nonlinearity is thus a priori ambiguous. However, in the previous case of Gaussian data in dimension 1, the behaviour of $r_\alpha$ (and then that of $u^\alpha$) has been proven to be sensibly different (\cite{nonlin_wave_mec, carlesgallagher, Ferriere__superposition_logNLS}).

\begin{prop}
    If $\lambda > 0$, then $r_\alpha$ is periodic. On the other hand, if $\lambda < 0$, then
    \begin{equation*}
        r_\alpha (t) \underset{t \rightarrow \infty}{\sim} 2t \sqrt{\abs{\lambda} \ln{t}}.
    \end{equation*}
\end{prop}

For $\lambda > 0$, such solutions $u^\alpha$ are almost periodic in time (up to a time-depending complex argument), which motivates to call them (and their derivates through the invariants and the scaling effect) \emph{breathers}. If those solutions are in dimension 1, they can be tensorized in order to find other solutions (also called \emph{breathers}) in higher dimension, even though they may be not periodic in general (see \cite{Ferriere__superposition_logNLS}).
However, in higher dimension $d \geq 2$, in the general case, the solutions \eqref{eq:gen_gauss_sol} are not periodic (and not "almost" periodic) and cannot be put under the form of a tensorization of breathers \eqref{eq:exp_breather_1d} in dimension 1 (already noticed in \cite{nonlin_wave_mec}). Therefore, all the functions \eqref{eq:exp_der_gen_gauss} will be called \emph{(general) Gaussian solutions} to \eqref{foc_log_nls} in the rest of the article, even though breathers are obviously a particular case of general Gaussian solutions.

Moreover, the ambiguity about the focusing or defocusing case has been removed in the general case by \cite{Cazenave_log_nls} (case $\lambda > 0$) and \cite{carlesgallagher} (case $\lambda < 0$). Indeed, in the latter, the authors show that all the solutions disperse in an unusually faster way (in the same way as for the Gaussian case) with a universal dynamic: after rescaling, the modulus of the solution converges to a universal Gaussian profile.
On the other hand, it has been proved that $\lambda > 0$ is the focusing case because there is no dispersion for large times thanks to the following result.
\begin{lem}[{\cite[Lemma~3.3]{Cazenave_log_nls}}]
    Let $\lambda > 0$.
    For any $k < \infty$ such that
    \begin{equation*}
        L_k \coloneqq \{ v \in W (\mathbb{R}^d), \norm{v}_{L^2} = 1, E(v) \leq k \} \neq \emptyset,
    \end{equation*}
    there holds
    \begin{equation*}
        \inf_{\substack{v \in L_k \\ 1 \leq p \leq \infty}} \norm{v}_{L^p} > 0.
    \end{equation*}
\end{lem}
This lemma, along with the conservation of the energy and the invariance through scaling factors (with Remark \ref{rem:scaling}), indicates that the solution to \eqref{foc_log_nls} is not dispersive, no matter how small the initial data are. For instance, its $L^\infty$ norm is bounded from below: to be more precise, there holds for all $t \in \mathbb{R}$ (see the proof of the above result)
\begin{equation*}
    \norm{u (t)}_{L^\infty} \geq \exp{\Bigl[ 1 - \frac{E (u(t))}{2 \lambda \, M (u(t))} \Bigr]} = \exp{\Bigl[ 1 - \frac{E (u_\textnormal{in})}{2 \lambda \, M (u_\textnormal{in})} \Bigr]}.
\end{equation*}
%

Actually, a specific Gaussian function \eqref{expr_gausson} called \emph{Gausson} and its derivates through the invariants of the equation and the scaling effect,
\begin{equation*}
    G^d_{\omega, x_0, v, \theta} (t,x) \coloneqq \exp \left[ i \left( \theta + 2 \lambda \omega t - v \cdot x + \frac{\abs{v}^2}{2} t \right) + \frac{d}{2} + \omega - \lambda \abs{x - x_0 - vt}^2 \right], \qquad t \in \mathbb{R}, \, x \in \mathbb{R}^d,
\end{equation*}
for any $\omega, \theta \in \mathbb{R}$, $x_0, v \in \mathbb{R}^d$, are known to be solutions to \eqref{foc_log_nls} for $\lambda > 0$, as proved in \cite{Dav_Mont_Squa_lognls} (and already noticed in \cite{nonlin_wave_mec}). It has also been proved that other radial stationary solutions to \eqref{foc_log_nls} exist (see \cite{Berestycki_Lions__excited_states_1d, Berestycki_Gallouet_Kavian__excited_state_2d, Dav_Mont_Squa_lognls}), but the Gausson is clearly special since it is the unique positive $\mathcal{C}^2$ stationary solution to \eqref{foc_log_nls} (also proved in \cite{Dav_Mont_Squa_lognls,Troy__Unique_pos_ground_state_logNLS}) and also since it is orbitally stable (\cite{Ardila__Orbital_stability_Gausson}, following the work of \cite{Cazenave_log_nls}).

\begin{theorem}[{\cite[Theorem~1.5]{Ardila__Orbital_stability_Gausson}}]
    Let $\omega \in \mathbb{R}$.
    For any $\varepsilon > 0$, there exists $\eta > 0$ such that for all $u_0 \in W (\mathbb{R}^d)$ satisfying
    \begin{equation*}
        \inf_{\theta, x_0} \norm{u_0 - e^{\omega + i \theta} G^d (. - x_0)}_{W ( \mathbb{R}^d )} < \eta,
    \end{equation*}
    the solution $u(t)$ of \eqref{foc_log_nls} with initial data $u_0$ satisfies
    \begin{equation*}
        \sup_t \inf_{\theta, x_0} \norm{u (t) - e^{\omega + i \theta} G^d (. - x_0)}_{W ( \mathbb{R}^d )} < \varepsilon.
    \end{equation*}
\end{theorem}

\begin{rem} \label{rem:link_br_sol}
    Remark that
    \begin{equation*}
        G^d_{\omega, x_0, v, \theta} \equiv B^{2 \lambda I_d}_{\omega, x_0, v, \theta},
    \end{equation*}
    where $I_d$ is the identity matrix in dimension $d$. Indeed, $2 \lambda I_d$ is a constant matrix solution to \eqref{eq:matrix_ODE}.
\end{rem}

\begin{rem}
    In particular, we point out that solitary wave solutions for \eqref{foc_log_nls} (\textit{i.e.} Gaussons) exist for ALL frequencies, unlike NLS equations with polynomial-like nonlinearity for which the only possible frequencies are (at least) non-negative. This is a consequence of the logarithmic nonlinearity, which satisfies $$g(s) \coloneqq - \lambda \ln{s} \underset{y \rightarrow 0}{\longrightarrow} + \infty,$$ unlike polynomial-like nonlinearity.
\end{rem}

\subsection{Main results}

\subsubsection{Existence of multi-Gaussons}

It was observed and proved for the Korteweg-de Vries equation that, for a large class of initial data, all solutions are global and eventually decompose into a finite sum of solitons going to the right and a dispersive part going to the left \cite{Eckhaus_Schuur__Soliton_Resolution_KdV, Schuur__Asymptotic_analysis_soliton}. This type of behavior is thought to be generic for nonlinear dispersive PDEs and this leads to the \emph{(Soliton) Resolution Conjecture}, which (vaguely formulated) states that any global solution of a nonlinear dispersive PDE will eventually decompose at large time as a combination of non-scattering structures (\textit{e.g.} a sum of solitary waves) and a radiative term.

Until recently, such conjecture had only been established for some integrable models, \textit{e.g.} the Korteweg-de Vries equation. The breakthrough approach introduced by Duyckaerts, Kenig and Merle allowed to prove this conjecture for some non-integrable equations such as the energy-critical wave equation \cite{Duyckaerts_Kenig_Merle__Class_radial_foc_wave} or the equivariant wave maps to the sphere \cite{Cote__soliton_resolution_equi_wave_sphere}. It remains an open problem for most of the classical nonlinear dispersive equations.

The Soliton Resolution Conjecture motivates the study of multi-soliton solutions for nonlinear dispersive PDE, \textit{i.e.} solutions which behave at large time as a sum of solitons. Indeed, investigating the existence and properties of solutions of dispersive equations made of a combination of non-scattering structures is a first step toward a proof of a Decomposition Conjecture, and multi-solitons are one of the simplest examples of a combination of non-scattering structures.

Several methods are available to obtain multi-solitons. They have been first constructed for NLS in the one dimensional cubic focusing case by Zakharov and Shabat \cite{Zakharov_Shabat__multisoliton_NLS_IST} using the inverse scattering transform method (IST). The IST is a powerful tool to study nonlinear dispersive equations and to exhibit non-trivial nonlinear dynamics for these equations. However, the IST application is restricted to equations which are completely integrable, like for example the Korteweg-de Vries equation and the cubic nonlinear Schrödinger equation in dimension 1. Moreover, integrability probably does not hold for \eqref{foc_log_nls}.

Another method to construct multi-soliton solutions of non-integrable equations was introduced by Martel, Merle and Tsai \cite{Martel_Merle_Tsai__Stability_Multisolitons_gKdV} for generalized Korteweg-de Vries equations and later developed in the case of $L^2$-subcritical nonlinear Schrödinger equations \cite{Martel_Merle__Multi_solitary_waves_NLS, Martel_Merle_Tsai__Stability_multisoliton_NLS}. This method uses tools usually called \emph{energy techniques}, in the sense that it relies on the use of the second variation of the energy as a Lyapunov functional to control the difference of a solution $u$ with the soliton sum $R$. It was later fine-tuned to allow the treatment of $L^2$ supercritical equations \cite{Cote_Martel_Merle_Construction_multisoliton_gKdV_NLS} and of profiles made with excited states \cite{Cote_LeCoz_Multisolitons_NLS}.

In this article, we show that this method for the construction of multi-solitons can be extended to a focusing logarithmic nonlinearity, revealing the existence of \emph{multi-Gaussons} for \eqref{foc_log_nls}. We will denote by $\mathcal{F}$ the Fourier Transform so that
\begin{equation*}
    \mathcal{F} (H^1) (\mathbb{R}^d) = \{ v \in L^2 (\mathbb{R}^d), \norm{\abs{x} \, v}_{L^2} < \infty \}
\end{equation*}
is a Hilbert space with its usual scalar product.

\begin{theorem}[Existence of multi-Gaussons] \label{th:exist_multi_sol}
Consider $\lambda>0$, $N \in \mathbb{N}^*$, $d \in \mathbb{N}^*$ and take $(v_k)_{1 \leq k \leq N}$ and $(x_k)_{1 \leq k \leq N}$ two families in $\mathbb{R}^d$, $(\omega_k)_{1 \leq k \leq N}$ and $(\theta_k)_{1 \leq k \leq N}$ two families of real numbers. Define 
\begin{equation*}
    v_* \coloneqq \min_{j \neq k} \, \abs{v_{j} - v_k}, \qquad \qquad
    G_k \coloneqq G^d_{\omega_k, x_k, v_k, \theta_k}.
\end{equation*}

If $v_* > 0$,
then there exist a unique solution $u \in \mathcal{C}_b (\mathbb{R}, W (\mathbb{R}^d)) \cap L^\infty_\textnormal{loc} (\mathbb{R}, \mathcal{F}(H^1) (\mathbb{R}^d))$ to \eqref{foc_log_nls} and $T \in \mathbb{R}$ such that $\forall t \geq 0$,
\begin{equation}
    \norm*[\Big]{u(T + t) - \sum_{k=1}^N G_k (T + t)}_{H^1 \cap \mathcal{F} (H^1)} \leq e^{- \frac{\lambda (v_* t)^2}{4}}. \label{dec_est_th}
\end{equation}
In particular, there exists $C > 0$ (depending on $\lambda$, $v_*$ and $T$) such that
\begin{equation*}
    \norm{u(t) - \sum_{k=1}^N G_k (t)}_{H^1 \cap \mathcal{F} (H^1)} \leq C \, e^{- \frac{\lambda (v_* t)^2}{8}}, \qquad \forall t \geq 0.
\end{equation*}
\end{theorem}

Several features are new for this case and should be pointed out.

First, for NLS with a nonlinearity of the form $g(\abs{u}^2) \, u$, the nonlinearity should usually satisfy a flatness property at 0 for such a result (for example in \cite{Cote_LeCoz_Multisolitons_NLS}: $g(0)=0$ and $\lim_{s \rightarrow 0} s \, g'(s) = 0$). Here, the nonlinearity in \eqref{foc_log_nls} is not flat at all at 0: even more, it is not even defined at 0 since we gave here $g = - \lambda \, \ln$ and
\begin{equation*}
    g(s) \underset{s \rightarrow 0^+}{\longrightarrow} + \infty, \qquad
    s \, g'(s) = \lambda \quad \forall s > 0.
\end{equation*}

The convergence rate of the solution to the sum of solitons is also very interesting. The convergence rate for NLS with polynomial-like nonlinearity is exponential, whereas it is "Gaussian-like" here, which is much faster. Moreover, it does not depend on the frequencies of the solitons anymore (even though $T$ does), unlike in \cite{Cote_LeCoz_Multisolitons_NLS}.

Such features may be surprising at first sight. However, they can be explained by the decay at infinity of the Gaussons. Indeed, in the same way as for the convergence rate, the decay of the solitons at infinity is usually exponential, with a rate depending on its frequency, whereas the solitons for \eqref{foc_log_nls} are the Gaussons, in particular Gaussian functions, whose decay at infinity is much faster and independent of their frequencies (up to a multiplicative constant). Moreover, the nonlinearity is still smooth enough: it is smooth far from the vacuum and, near the vacuum, $u \ln{\abs{u}^2}$ remains almost lipschitz.

Remark also that the convergence is in $H^1 (\mathbb{R}^d) \cap \mathcal{F} (H^1 (\mathbb{R}^d))$. Hence the same convergence rate holds in the energy space $W (\mathbb{R}^d)$ since it is known that 
\begin{equation*} 
    H^1 (\mathbb{R}^d) \cap \mathcal{F} (H^1 (\mathbb{R}^d)) \subset W (\mathbb{R}^d).
\end{equation*}
We can compare this to the case of a subcritical nonlinearity, where the convergence is also proved in the energy space for such an equation, \textit{i.e.} "only" $H^1$.

\begin{rem} \label{rem:rigidity}
    Theorem \ref{th:exist_multi_sol} does not say that, as soon as we fixed all the parameters, the multi-Gausson is unique. However, there is a unique multi-Gausson which satisfies the convergence rate property \eqref{dec_est_th} (for those parameters). Indeed, any other multi-Gausson $v$ would satisfy for all $t \geq T_1$
    \begin{equation*}
        \norm{v(t) - \sum_{k=1}^N G_k (t)}_{L^2} \geq C_1 \, e^{- 2 \lambda t}.
    \end{equation*}
    for some constant $C_1 > 0$ and some time $T_1$ (see Lemma \ref{lem:rigidity_multi}). Therefore, it is a \emph{rigidity} property.
\end{rem}

\subsubsection{Existence of multi-breathers and multi-gaussians}

The Gaussons are not the only non-scattering structures that we are aware of for this equation: we have excited states, but we also have breathers and more generally gaussian solutions. Thus, the (Soliton) Resolution Conjecture also motivates the study of multi-gaussians. However, those breathers and gaussian solutions are \emph{not} bound state for the energy $E$, and then the same energy techniques cannot be applied. Nevertheless, the method used to find the $L^2$ estimate for the multi-Gaussons for \eqref{foc_log_nls} does not involve the energy: such a method can be tuned in order to fit with these multi-gaussians.

\begin{theorem}[Existence of multi-gaussians] \label{th:exist_multi_br}
Consider $\lambda>0$, $N \in \mathbb{N}^*$, $d \in \mathbb{N}^*$ and take $(v_k)_{1 \leq k \leq N}$ and $(x_k)_{1 \leq k \leq N}$ two families in $\mathbb{R}^d$, $(\omega_k)_{1 \leq k \leq N}$ and $(\theta_k)_{1 \leq k \leq N}$ two families of real numbers, and $(A_k^\textnormal{in})_{1 \leq k \leq N}$ a sequence of complex matrices in $S_d (\mathbb{C})^{\Re +}$. Define $A_k (t)$ the solution to \eqref{eq:matrix_ODE} with initial data $A_k^\textnormal{in}$ and
\begin{equation*}
    v_* \coloneqq \min_{j \neq k} \, \abs{v_{j} - v_k}, \qquad \qquad
    B_k \coloneqq B^{A_k^\textnormal{in}}_{\omega_k, x_k, v_k, \theta_k}, \qquad \qquad
    \sigma_- \coloneqq \frac{1}{2} \inf_{t,k} \sigma ( \Re A_k (t) ) > 0.
\end{equation*}

If $v_* > 0$,
then there exist a unique solution $u \in \mathcal{C}_b (\mathbb{R}, W (\mathbb{R}^d))$ to \eqref{foc_log_nls} and $T \in \mathbb{R}$ such that $\forall t \geq 0$,
\begin{equation}
    \norm*[\Big]{u(T + t) - \sum_{k=1}^N B_k (T + t)}_{L^2} \leq e^{- \frac{\sigma_- (v_* t)^2}{4}}. \label{eq:dec_est_th_br}
\end{equation}
In particular, there exists $C > 0$ (depending on $\lambda$, $v_*$ and $T$) such that
\begin{equation*}
    \norm{u(t) - \sum_{k=1}^N B_k (t)}_{L^2} \leq C \, e^{- \frac{\sigma_- (v_* t)^2}{8}}, \qquad \forall t \geq 0.
\end{equation*}
\end{theorem}

\begin{rem}
    Remark that the convergence is only in $L^2$ norm here, unlike the previous theorem for multi-Gaussons where the convergence is in $H^1 \cap \mathcal{F}(H^1)$. However, we do believe that a convergence in $H^1 \cap \mathcal{F}(H^1)$ should hold, but the energy techniques for such a proof do not hold, as already pointed out.
\end{rem}

\begin{rem} \label{rem:rigidity_br}
    Again, the "uniqueness" of the multi-gaussians is subjected to the convergence rate property \eqref{eq:dec_est_th_br}: any other solution $v$ would satisfy for all $t \geq T_1$
    \begin{equation*}
        \norm{v(t) - \sum_{k=1}^N B_k (t)}_{L^2} \geq C_1 e^{- 2 \lambda t}.
    \end{equation*}
    for some constant $C_1 > 0$ and some time $T_1$ (see Lemma \ref{lem:rigidity_multi}) in the same way.
\end{rem}

\subsection{Scheme of the proof and outline}

Our strategy for the proofs of Theorems \ref{th:exist_multi_sol} and \ref{th:exist_multi_br} is inspired from the works \cite{Martel_Merle__Multi_solitary_waves_NLS, Merle_Construction_blow_up_points, Cote_Martel_Merle_Construction_multisoliton_gKdV_NLS, Cote_LeCoz_Multisolitons_NLS}: we take a sequence of time $T_n \rightarrow + \infty$ and a set of final data $u_n (T_n) = B (T_n)$ where $B \coloneqq \sum_{k=1}^N B_k$ (in the case of multi-Gaussons, $A_k^\textnormal{in} = 2 \lambda I_d$ like already pointed out in Remark \ref{rem:link_br_sol}). Our goal is to prove that the solutions $u_n$ to \eqref{foc_log_nls} (which approximate a multi-soliton) enjoy uniform $H^1 (\mathbb{R}^d)$ and $\mathcal{F}(H^1 (\mathbb{R}^d))$ (in the case of multi-Gaussons) or only $L^2 (\mathbb{R}^d)$ (in the more general case) decay estimates on $[T, T_n]$ for some $T$ independent of $n$. Then, a compactness in $L^2$ is proved which shows that $(u_n)$ (up to a subsequence) converges to a multi-soliton solution to \eqref{foc_log_nls}.

For multi-Gaussons, like in \cite{Martel_Merle__Multi_solitary_waves_NLS, Cote_Martel_Merle_Construction_multisoliton_gKdV_NLS, Cote_LeCoz_Multisolitons_NLS}, the uniform backward $H^1 (\mathbb{R}^d)$-estimates rely on slow variation of localized conservation laws as well as an $H^1 (\mathbb{R}^d)$ coercivity of the action around $G$ up to an $L^2 (\mathbb{R}^d)$-norm and is proved through a bootstrap property. The main new difficulty for \eqref{foc_log_nls} compared to a subcritical polynomial-like nonlinearity is that the energy $E$ on which the action is constructed is not of class $\mathcal{C}^2$ because of the potential energy. However, a weaker Taylor expansion holds, and thus the coercivity (in $H^1 (\mathbb{R}^d)$ up to an $L^2 (\mathbb{R}^d)$ norm) can still be proved.

Another new feature of the proof is that the uniform $L^2 (\mathbb{R}^d)$-estimates can be found thanks to a more direct computation from \cite{Ferriere__superposition_logNLS}, inspired from \cite{cazenave-haraux, carlesgallagher}, 
performed in Section \ref{sec:L2_est} (both for the multi-gaussian and the multi-Gausson cases). Thus, for the Gaussons case, the bootstrap property is needed only for the homogeneous $\dot{H}^1 (\mathbb{R}^d)$-norm. After proving the uniform $H^1$-estimates property in Section \ref{sec:H1_est}, a similar improved computation as that for the uniform $L^2 (\mathbb{R}^d)$-estimate is performed for the uniform $\mathcal{F} (H^1 (\mathbb{R}^d))$ estimates in Section \ref{sec:FH1_est}, which also gives compactness in $L^2 (\mathbb{R}^d)$.

As for the multi-gaussian case, the compactness property is proved in Section \ref{sec:compactness_multi_br} thanks to a virial argument, similar to that in \cite{Martel_Merle__Multi_solitary_waves_NLS, Cote_Martel_Merle_Construction_multisoliton_gKdV_NLS, Cote_LeCoz_Multisolitons_NLS}. Eventually, the rigidity property for both cases is proved in Section \ref{sec:rigidity} thanks to a consequence of the $L^2$ energy estimate (Lemma \ref{lem:L2_energy_est}).

\subsection*{Notation}

From now on, $C$ will denote a positive constant which does not depend on anything and $C_0$ a positive constant which is independent of the time and of $n$ (but may depend on other parameters). They also may change from line to line.
Moreover, all the functional spaces in space are in $\mathbb{R}^d$, which will be implicit. For instance, we will denote by $L^2$, $H^1$, $\mathcal{F}(H^1)$ and $W$ instead of $L^2 (\mathbb{R}^d)$, $H^1 (\mathbb{R}^d)$, $\mathcal{F}(H^1 (\mathbb{R}^d))$ and $W (\mathbb{R}^d)$.
Furthermore, all the integrals will be in $\mathbb{R}^d$ and all the sums will be from $1$ to $N$, except if indicated.

\section{Uniform $L^2$-estimates} \label{sec:L2_est}

\subsection{Approximate solutions and convergence toward a multi-soliton}

As already said, the proof follow the general scheme laid down by Martel, Merle and Tsai \cite{Martel_Merle_Tsai__Stability_Multisolitons_gKdV} for the Korteweg-de Vries equation and adapted by Martel and Merle \cite{Martel_Merle__Multi_solitary_waves_NLS} in the case of nonlinear Schrödinger equation (we can also cite for instance the works \cite{Merle_Construction_blow_up_points, Cote_Martel_Merle_Construction_multisoliton_gKdV_NLS, Cote_LeCoz_Multisolitons_NLS}).
We choose an increasing sequence of times $(T_n)_{n\in\mathbb{N}}$ with $T_n \rightarrow \infty$ as $n \rightarrow \infty$ and we define solutions $u_n$ to \eqref{foc_log_nls} with \emph{final data} $u_n (T_n) = B (T_n)$ where $B \coloneqq \sum B_k$ with $B_k \coloneqq B^{A_k^\textnormal{in}}_{\omega_k, x_k, v_k, \theta_k}$. In the Gaussons case, we take $A_k^\textnormal{in} = 2 \lambda I_d$ so that $B^{A_k^\textnormal{in}}_{\omega_k, x_k, v_k, \theta_k} = G^d_{\omega_k, x_k, v_k, \theta_k}$ for all $k$ (see Remark \ref{rem:link_br_sol}), and we may also say $G_k$ instead of $B_k$ and $G$ instead of $B$. Thanks to Theorem \ref{th:Cauchy_problem}, we know that all $u_n$ are well-defined and global since $B (T_n)$ is obviously in $W$, so $u_n \in \mathcal{C}_b (\mathbb{R}, W)$.

Our goal is to prove that $u_n$ (called \emph{approximate multi-gaussian}, resp. \emph{approximate multi-soliton} in the Gausson case) converges (up to a subsequence) to $u$, a multi-gaussian (resp. multi-soliton) solution to \eqref{foc_log_nls} which satisfies the estimates of Theorem \ref{th:exist_multi_br} (resp. \ref{th:exist_multi_sol}). For this, we prove that the $u_n$s satisfy the same kind of decay estimate as in \eqref{eq:dec_est_th_br} (resp. \eqref{dec_est_th}) but only up to $T_n$.

\subsubsection{Multi-gaussian case}

As in \cite{Martel_Merle__Multi_solitary_waves_NLS, Cote_LeCoz_Multisolitons_NLS}, Theorem \ref{th:exist_multi_br} relies on two important propositions:

\begin{prop}[Uniform Estimates] \label{prop:unif_est_br}
There exists $T \in \mathbb{R}$ such that for all $n \in \mathbb{N}$ such that $T_n > T$ and for any $t \in [0, T_n - T]$,
\begin{equation}
    \norm{u_n (T + t) - B (T + t)}_{L^2} \leq e^{- \frac{\sigma_- (v_* t)^2}{4}}. \label{eq:L2_unif_est_br}
\end{equation}
\end{prop}

\begin{prop}[Compactness] \label{prop:compactness_br}
There exists $u_\textnormal{in} \in W$ such that (up to a subsequence)
\begin{equation*}
    \lim_{n \rightarrow \infty} \norm{u_n (T) - u_\textnormal{in}}_{L^2} = 0.
\end{equation*}
\end{prop}

Theorem \ref{th:exist_multi_br} can then be easily proved thanks to Propositions \ref{prop:unif_est_br} and \ref{prop:compactness_br}.

\begin{proof}[Proof of Theorem \ref{th:exist_multi_br}]
    Let $u_\textnormal{in}$ given by Corollary \ref{prop:compactness_br}.
    Take the subsequence of $(u_n)$ (still denoted $(u_n)$) such that $(u_n (T))$ converges to $u_\textnormal{in}$ in $L^2$ as $n \rightarrow \infty$ and let $u \in \mathcal{C}_b (\mathbb{R}, W)$ be the solution to \eqref{foc_log_nls} with initial data $u(T) = u_\textnormal{in} \in W$ given by Theorem \ref{th:Cauchy_problem}.
    Then, thanks to Lemma \ref{lem:L2_energy_est}, $\mu_n \coloneqq u - u_n$ satisfies for all $t \geq 0$ and $n \in \mathbb{N}$,
    \begin{equation*}
        \norm{\mu_n (T + t)}_{L^2} \leq \norm{\mu_n (T)}_{L^2} \, e^{2 \lambda t}.
    \end{equation*}
    By definition of $\mu_n$ and $u$, $\mu_n (T) = u_n (T) - u_\textnormal{in} \rightarrow 0$ in $L^2$ as $n \rightarrow \infty$. Therefore, for all $t \geq 0$, the previous inequality shows that $u_n (T + t) \rightarrow u (T + t)$ in $L^2$ as $n \rightarrow \infty$.
    Then, using Proposition \ref{prop:unif_est_br} and taking the limit $n \rightarrow \infty$ for any $t \geq 0$, we get:
    \begin{equation*}
        \norm{u (T + t) - B (T + t)}_{L^2} \leq \lim_{n \rightarrow \infty} \norm{u_n (T + t) - B (T + t)}_{L^2} \leq e^{- \frac{\sigma_- (v_* t)^2}{4}}. \qedhere
    \end{equation*}
\end{proof}

The Uniform Estimates property \ref{prop:unif_est_br} will be proved in Subsection \ref{subsec:L2_est}, while Section \ref{sec:compactness_multi_br} will be devoted to the proof of the Compactness property \ref{prop:compactness_br}. However, before that, we also look at the multi-soliton case.

\subsubsection{Multi-soliton case}

For the multi-soliton case, the same kind of properties will be proved, but the uniform estimates are in $H^1 \cap \mathcal{F}(H^1)$ norm, not only $L^2$.

\begin{prop}[Uniform Estimates] \label{prop:unif_est}
There exists $T \in \mathbb{R}$ such that for all $n \in \mathbb{N}$ such that $T_n > T$ and for any $t \in [0, T_n - T]$,
\begin{equation}
    \norm{u_n (T + t) - G (T + t)}_{H^1 \cap \mathcal{F} (H^1)} \leq e^{- \frac{\lambda (v_* t)^2}{4}}. \label{eq:H1_FH1_unif_est}
\end{equation}
\end{prop}

Like for the multi-gaussian case, we also need a compactness property. However, it obviously results from the uniform estimates:

\begin{cor}[Compactness] \label{cor:compactness}
There exists $u_\textnormal{in} \in H^1 \cap \mathcal{F}(H^1)$ such that (up to a subsequence)
\begin{equation*}
    \lim_{n \rightarrow \infty} \norm{u_n (T) - u_\textnormal{in}}_{L^2} = 0.
\end{equation*}
\end{cor}

\begin{proof}[Proof of Corollary \ref{cor:compactness}]
    $u_n (T)$ is uniformly bounded in $H^1 \cap \mathcal{F}(H^1)$ by taking $t=0$ in \eqref{eq:H1_FH1_unif_est}, which yields the conclusion since the embedding $H^1 \cap \mathcal{F}(H^1) \subset L^2$ is compact.
\end{proof}

\begin{rem}
    Unlike this case, the compactness property is not as obvious in \cite{Martel_Merle__Multi_solitary_waves_NLS, Cote_LeCoz_Multisolitons_NLS} (or for the multi-gaussian case above). Indeed, the uniform estimates are only in $H^1$ there (or even in $L^2$ for the multi-gaussian), whereas we also have $\mathcal{F}(H^1)$ here. Thus, the authors had to prove in there a uniform equicontinuity of the sequence $(u_n)$ by using a virial argument and the uniform estimates in $L^2$. The same kind of proof will be used for Proposition \ref{prop:compactness_br}.
\end{rem}

Theorem \ref{th:exist_multi_sol} is then a corollary of Proposition \ref{prop:unif_est} and Corollary \ref{cor:compactness}. Its proof is totally similar to the proof of Theorem \ref{th:exist_multi_br}, using the weakly lower semi-continuity of the $H^1 \cap \mathcal{F} (H^1)$-norm.
%
%
%
For this case, the uniform estimates in $L^2$ will also be proved in Subsection \ref{subsec:L2_est}, in the same time as for the multi-gaussian case. As for the $H^1$ and the $\mathcal{F}(H^1)$ uniform estimates, they will be proved in Sections \ref{sec:H1_est} and \ref{sec:FH1_est} respectively.

\subsection{Uniform $L^2$-estimates} \label{subsec:L2_est}

In \cite{Martel_Merle__Multi_solitary_waves_NLS, Cote_Martel_Merle_Construction_multisoliton_gKdV_NLS, Cote_LeCoz_Multisolitons_NLS}, the uniform estimates in $H^1$ for multi-solitons are found thanks to a bootstrap argument. Here, the uniform estimates in $L^2$ can be found without it, directly with a rough stability estimate thanks to the strong decay of our solitons at infinity. This method can also be extended to the gaussian case, and we prove both cases in one property:

\begin{prop}[Uniform estimates in $L^2$] \label{prop:unif_est_L2_br}
There exists $T' \in \mathbb{R}$ such that for all $n \in \mathbb{N}$ such that $T_n > T'$ and for any $t \in [0, T_n - T']$,
\begin{equation*}
    \norm{u_n (T' + t) - B (T' + t)}_{L^2} \leq e^{- \frac{\sigma_- (v_* t)^2}{4}}. \label{eq:unif_est_L2_br}
\end{equation*}
\end{prop}

\begin{rem}
    We recall that the multi-soliton case is a particular case of multi-gaussian where we have taken $A_k^\textnormal{in} = 2 \lambda I_d$, so that $A_k (t) = 2 \lambda I_d$. Thus, by definition of $\sigma_-$, we get $\sigma_- = \lambda$, which gives exactly the expected convergence rate for multi-Gausson, at least for the $L^2$ norm.
\end{rem}{}

The computation used for this estimate is almost the same computation as that in \cite{Ferriere__superposition_logNLS}. Directly inspired from the computation for the energy estimate in $L^2$ found in \cite{cazenave-haraux}, it is the consequence of the following lemma:

\begin{lem}[{\cite[Lemma~1.1.1]{cazenave-haraux}}] \label{lem_log_inequality}
There holds
\begin{equation*}
    \abs{\Im \left( (z_2 \ln \abs{z_2}^2 - z_1 \ln \abs{z_1}^2) (\overline{z_2} - \overline{z_1}) \right)} \leq 2 \abs{z_2 - z_1}^2, \qquad \forall z_1, z_2 \in \mathbb{C}.
\end{equation*}
\end{lem}

Indeed, for any integer $N \geq 1$ and any solutions $v, v_1, \dots, v_N$ to \eqref{foc_log_nls}, the function $w \coloneqq v - V$ with $V \coloneqq \sum v_j$ satisfies:
\begin{equation*}
    i \, \partial_t w + \frac{1}{2} \Delta w = - \lambda \Bigl( v \ln{\abs{v}^2} - \sum v_j \ln{\abs{v_j}^2} \Bigr).
\end{equation*}
Therefore, there holds
\begin{align*}
    \frac{1}{2} \frac{\diff}{\diff t} \norm{w (t)}_{L^2}^2 &= - \lambda \Im \int \Bigl( v \ln{\abs{v}^2} - \sum v_j \ln{\abs{v_j}^2} \Bigr) ( \overline{v} - \overline{V} ) \diff x \\
        &= - \lambda \Im \int \Bigl( v \ln{\abs{v}^2} - V \ln{\abs{V}^2} \Bigr) ( \overline{v} - \overline{V} ) \diff x - \lambda \Im \int \Bigl( V \ln{\abs{V}^2} - \sum v_j \ln{\abs{v_j}^2} \Bigr) ( \overline{v} - \overline{V} ) \diff x \\
    \frac{1}{2} \abs{\frac{\diff}{\diff t} \norm{w (t)}_{L^2}^2} &\leq 2 \abs{\lambda} \, \norm{w}_{L^2}^2 + \abs{\lambda} \int \abs{ V \ln{\abs{V}^2} - \sum v_j \ln{\abs{v_j}^2} } \abs{w} \diff x \\
        &\leq 2 \abs{\lambda} \, \norm{w}_{L^2}^2 + \abs{\lambda} \norm{ V \ln{\abs{V}^2} - \sum v_j \ln{\abs{v_j}^2} }_{L^2} \norm{w}_{L^2}.
\end{align*}
Dividing by $\norm{w}_{L^2}$, we obtain the inequality:
\begin{equation}
    \abs{\frac{\diff}{\diff t} \norm{w (t)}_{L^2}} \leq 2 \abs{\lambda} \, \norm{w}_{L^2} + \abs{\lambda} \norm{ V \ln{\abs{V}^2} - \sum_{j=1}^N v_j \ln{\abs{v_j}^2} }_{L^2}. \label{eq:energy_est_like}
\end{equation}

The core of the problem is to estimate the last term in order to be able to perform a backward Gronwall lemma. In the case where $v_j = B_j$, this is an explicit term which can be estimated thanks to \cite[Lemma~3.2]{Ferriere__superposition_logNLS} that we recall here.

\begin{lem} \label{lem:diff_L_log_L_sum_gaussian}
For any $d \in \mathbb{N}^*$, there exists $C_d > 0$ such that the following holds.
Let $N \in \mathbb{N}^*$ and take $x_k \in \mathbb{R}^d$, $\omega_k \in \mathbb{R}$, $\Lambda_k \in S_d (\mathbb{C})^{\Re +}$ and $\theta_k : \mathbb{R}^d \rightarrow \mathbb{R}$ a real measurable function for $k = 1, \dots, N$, and define for all $x \in \mathbb{R}^d$
\begin{equation*}
    g_k (x) = \exp \left[ i \theta_k (x) + \omega_k - (x - x_k)^\top \Lambda_k (x - x_k) \right],
\end{equation*}
as well as
\begin{equation*}
    g (x) = \sum_{k = 1,\dots,N} g_k (x).
\end{equation*}
If
\begin{equation*}
    \varepsilon \coloneqq \left( \min_{k \neq j} \, \abs{x_{j} - x_k} \right)^{-1} < \varepsilon_0 \coloneqq  \min \biggl( \frac{\sqrt{\lambda_+}}{\max (\sqrt{\delta \omega + 1}, \sqrt{\ln{N}})}, \sqrt{\frac{\lambda_-}{d+2}} \biggr)
\end{equation*}
where $\delta \omega \coloneqq \max\limits_{j,k} |\omega_k - \omega_j|$, $\lambda_+ = \max\limits_k \Re \sigma (\Lambda_k)$ and $\lambda_- = \min\limits_k \Re \sigma (\Lambda_k) > 0$, then
\begin{equation} \label{L2_norm_gauss_log}
    \norm{g \ln \abs{g} - \sum_{k = 1}^N g_k \ln \abs{g_k}}_{L^2} \leq C_d N^\frac{3}{2} \, \frac{\lambda_+}{\varepsilon^{\frac{d}{2}+1} \sqrt{\lambda_-}} \, \exp \left[ - \frac{\lambda_-}{4 \varepsilon^2} + \max_j \omega_j \right].
\end{equation}
\end{lem}

To be able to apply Lemma \ref{lem:diff_L_log_L_sum_gaussian}, we need the gaussians to be "well separated" (which is given by the condition $\varepsilon < \varepsilon_0$). But this happens, eventually after some time 
\begin{equation*}
    T_\textnormal{sep} \coloneqq \max_{k, j} \frac{\varepsilon_0^{-1} + \abs{x_k - x_j}}{v_*},
\end{equation*}
where $\varepsilon_0$ is defined in Lemma \ref{lem:diff_L_log_L_sum_gaussian}.
Indeed, for all times $t \geq 0$, there holds
\begin{equation}
    \abs{x_j + v_j (T_\textnormal{sep} + t) - (x_k + v_k (T_\textnormal{sep} + t))} \geq \varepsilon_0^{-1} + v_* t. \label{eq:distance_x_k}
\end{equation}
Then, we will also need to estimate the Gauss error function, which can be done in the usual way.

\begin{lem} \label{lem:int_error_gauss}
    For any $y > 0$ and $\gamma > 0$, there holds
    \begin{gather*}
        \int_y^\infty e^{-\gamma x^2} \diff x < \frac{1}{2 \gamma y} e^{-\gamma y^2}.
    \end{gather*}
\end{lem}

\begin{proof}[Proof of Lemma \ref{lem:int_error_gauss}]
    We easily compute:
    \begin{equation*}
        \int_y^\infty e^{-\gamma x^2} \diff x < \int_y^\infty \frac{x}{y} e^{-\gamma x^2} \diff x = \frac{1}{y} \left[ - \frac{e^{-\gamma x^2}}{2 \gamma} \right]^\infty_y = \frac{1}{2 \gamma y} e^{-\gamma y^2}. \qedhere
    \end{equation*}
\end{proof}

\begin{proof}[Proof of Proposition \ref{prop:unif_est_L2_br}]
Thanks to \eqref{eq:distance_x_k} and the fact that $0 < \sigma_- = \inf_{t,k} \sigma ( \Re A_k (t) ) \leq \sup_{t,k} \sigma ( \Re A_k (t) ) < + \infty$ with Proposition \ref{prop:expression_general_gaussian}, we can apply Lemma \ref{lem:diff_L_log_L_sum_gaussian} : for all $t > 0$,
\begin{multline}
    \norm{B (T_\textnormal{sep} + t) \ln \abs{B (T_\textnormal{sep} + t)} - \sum_{k = 1}^N B_k (T_\textnormal{sep} + t) \ln \abs{B_k (T_\textnormal{sep} + t)}}_{L^2 (\mathbb{R})} \\ 
        \leq C_0 (\varepsilon_0^{-1} + v_* t)^\frac{d+2}{2} \exp \left[ - \frac{\sigma_- (\varepsilon_0^{-1} + v_* t)^2}{4} \right] 
        \leq C_0 \, \exp \left[ - \frac{\sigma_- (v_* t)^2}{4} \right]. 
    \label{est_g_ln_g_L2_multisoliton}
\end{multline}
Take $n$ large enough so that $T_n > T_\textnormal{sep}$ and set $w_n \coloneqq u_n - G$. \eqref{eq:energy_est_like} gives that for all $t > 0$,
\begin{equation*}
    \abs*[\Big]{\partial_t \norm{w_n (T_\textnormal{sep} + t)}_{L^2}} \leq 2 \lambda \norm{w_n (T_\textnormal{sep} + t)}_{L^2} + C_0 \, \exp \left[ - \frac{\sigma_- (v_* t)^2}{4} \right].
\end{equation*}
The Gronwall lemma (backward in time) between $t$ and $T_n - T_\textnormal{sep}$ for $0 \leq t \leq T_n - T_\textnormal{sep}$ and the fact that $w_n (T_n) = 0$ yields
\begin{align*}
    \norm{w_n (T_\textnormal{sep} + t)}_{L^2} &\leq C_0 \int_t^{T_n} \exp \left[ - \frac{\sigma_- (v_* s)^2}{4} + 2 \lambda (s-t) \right] \diff s \\
        &\leq C_0 \, e^{-2 \lambda t} \int_t^{\infty} \exp \left[ - \frac{\sigma_- (v_* s)^2}{4} + 2 \lambda s \right] \diff s.
\end{align*}
Then, we estimate the integral.
\begin{align*}
    \int_{t}^{\infty} \exp \left[ - \frac{\sigma_- (v_* s)^2}{4} + 2 \lambda s \right] \diff s &= \int_{t}^{\infty} \exp \left[ - \sigma_- \left( \frac{v_* s}{2} - \frac{2 \lambda}{\sigma_- v_*} \right)^2 + \frac{4 \lambda^2}{\sigma_- v_*^2} \right] \diff s \\
        &= \frac{2}{v_*} \int_{\tilde{t}}^{\infty} \exp \left[ - \sigma_- r^2 + \frac{4 \lambda^2}{\sigma_- v_*^2} \right] \diff r,
\end{align*}
where $\tilde{t} \coloneqq \frac{v_* t}{2} - \frac{2 \lambda}{\sigma_- v_*}$.
Using Lemma \ref{lem:int_error_gauss}, we get for all $t$ such that $\tilde{t} = \frac{v_* t}{2} - \frac{2 \lambda}{\sigma_- v_*} > 0$
\begin{equation*}
    \int_{t}^{\infty} \exp \left[ - \frac{\sigma_- (v_* s)^2}{4} + 2 \lambda s \right] \diff s \leq \frac{C_0}{\tilde{t}} \exp \left[ - \sigma_- \tilde{t}^2 + \frac{4 \lambda^2}{\sigma_- v_*^2} \right] = \frac{C_0}{\tilde{t}} \exp \left[ - \frac{\sigma_- (v_* t)^2}{4} + 2 \lambda t \right]
\end{equation*}
Hence, setting $t_1 \coloneqq \frac{4}{v_*^2}$, we have for all $n$ large enough and for all $t \in [t_1, T_n - T_\textnormal{sep}]$
\begin{equation*}
    \norm{w_n (T_\textnormal{sep} + t)}_{L^2} \leq \frac{C_0}{t - t_1} \exp \left[ - \frac{\sigma_- (v_* t)^2}{4} \right],
\end{equation*}
which leads to the result by taking $T' > T_\textnormal{sep} + t_1$ large enough.
\end{proof}

\section{Uniform $H^1$-estimates} \label{sec:H1_est}

This section (and so is Section \ref{sec:FH1_est}) is completely devoted to the multi-soliton case, \textit{i.e.} Proposition \ref{prop:unif_est}, so that here $B_j = G_j$ are all Gaussons. However, since the Gausson is a particular case of gaussian solutions, the previous section holds. Moreover, for Gaussons, we have $A_j (t) = 2 \lambda I_d$, thus $\sigma_- = \lambda$ here. Therefore, Proposition \ref{prop:unif_est_L2_br} gives here:

\begin{prop}[Uniform estimates in $L^2$, multi-Gausson case] \label{prop:unif_est_L2}
There exists $T' \in \mathbb{R}$ such that for all $n \in \mathbb{N}$ such that $T_n > T'$ and for any $t \in [0, T_n - T']$,
\begin{equation*}
    \norm{u_n (T' + t) - G (T' + t)}_{L^2} \leq e^{- \frac{\lambda (v_* t)^2}{4}}. \label{eq:unif_est_L2}
\end{equation*}
\end{prop}

The second step in order to prove Proposition \ref{prop:unif_est} is to get the uniform estimates in $H^1$. Since we already have it for $L^2$ (and we fix $T'$ provided by Proposition \ref{prop:unif_est_L2}), we need to prove it only for $\dot{H}^1$.

\begin{prop} \label{prop:unif_est_dot_H1}
    There exists $T > T'$ such that for all $n \in \mathbb{N}$ such that $T_n > T$ and for any $t \in [0, T_n - T]$,
    \begin{equation}
        \norm{u_n (T + t,.) - G (T + t,.)}_{\dot{H}^1} \leq e^{- \frac{\lambda (v_* t)^2}{4}}. \label{eq:dot_H1_unif_est}
    \end{equation}
\end{prop}

The proof relies on a bootstrap argument. Indeed, from the definition of the final data $u_n(T_n)$ and continuity of $u_n$ in time with values in $H^1$, it follows that \eqref{eq:dot_H1_unif_est} holds on an interval $[t^\dag,T_n]$ for $t^\dag$ close enough to $T_n$. Then the following Proposition \ref{prop:bootstrap} shows that we can actually improve it to a better estimate, hence leaving enough room to extend the interval on which the original estimate holds.

\begin{prop}[Bootstrap Property] \label{prop:bootstrap}
    There exists  $T'' > T'$ and $t_* > 0$ such that for all $n \in \mathbb{R}$ such that $T_n > T'' + t_*$ and for all $t^\dag \in [t_*, T_n - T'']$, the following holds. If for all $t \in [t^\dag, T_n - T'']$ we have
    \begin{equation}
        \norm{u_n (T'' + t) - G (T'' + t)}_{\dot{H}^1} \leq e^{- \frac{\lambda (v_* t)^2}{4}}, \label{eq:bootstrap_hyp}
    \end{equation}
    then for all $t \in [t^\dag, T_n - T'']$ there holds
    \begin{equation*}
        \norm{u_n (T'' + t) - G (T'' + t)}_{\dot{H}^1} \leq \frac{1}{2} e^{- \frac{\lambda (v_* t)^2}{4}}.
    \end{equation*}
\end{prop}

\begin{rem}
    To be more precise, what we prove is the following: for all $t^\dag \in [0, T_n - T'']$, if for all $t \in [t^\dag, T_n - T'']$ we have \eqref{eq:bootstrap_hyp}, then for all $t \in [t^\dag, T_n - T'']$ there holds
    \begin{equation*}
        \norm{u_n (T'' + t) - G (T'' + t)}_{\dot{H}^1}^2 \leq C_0 \, t^{-1} e^{- \frac{\lambda (v_* t)^2}{2}}.
    \end{equation*}
    Therefore, we need some $t_* > 0$ so that $\frac{C_0}{t} \leq \frac{1}{4}$ for all $t \geq t_*$ in order to have a true bootstrap argument.
\end{rem}

In \cite{Martel_Merle__Multi_solitary_waves_NLS, Cote_Martel_Merle_Construction_multisoliton_gKdV_NLS, Cote_LeCoz_Multisolitons_NLS}, the bootstrap argument to get the uniform estimates in $H^1$ uses a modified action defined with localized quantities (localization of the conserved quantities around each member/soliton). It is known that each soliton gives only at most few negative $L^2$ "bad directions" in the Hessian of the action and has an exponential decay at infinity. Therefore the Hessian of this modified action is "almost coercive", up to an exponentially decreasing function of time, and to some bad directions in $L^2$ which are controlled either by modulation (\cite{Martel_Merle__Multi_solitary_waves_NLS} for instance) or by controlling the $L^2$-growth without the help of the Hessian (\cite{Cote_LeCoz_Multisolitons_NLS} for instance). Therefore one gets a slightly better estimate in $H^1$ than provided by the assumption.

Here, such an argument cannot be used directly: the energy is not $\mathcal{C}^2$ because of the appearance of a $\ln \abs{v}$ term from the potential energy, which is ill-posed in $\mathcal{L} (L^2)$, and therefore we cannot have a Taylor expansion relation between the modified action and its linearized.
However, a weaker expansion of the potential energy (coming from Lemma \ref{lem:exp_u_ln_u}) can still be used, giving enough room to get an $H^1$-coercivity up to an (almost) $L^2$ norm, as shown in Lemma \ref{lem:linearize_action}.

From this proposition, the proof of Proposition \ref{prop:unif_est_dot_H1} is then completely similar to those in \cite{Martel_Merle__Multi_solitary_waves_NLS, Cote_LeCoz_Multisolitons_NLS}, and we refer to them for more precision.

\subsection{The Bootstrap Property}

The idea of the proof of the Bootstrap Property \ref{prop:bootstrap} is similar to that of \cite{Cote_LeCoz_Multisolitons_NLS}, which is reminiscent of the technique used to prove stability for a single soliton in the subcritical case (see \textit{e.g.} \cite{Weinstein:Stability_ground_state_NLS, Weinstein:Stability_ground_states_disp_eq, Grillakis_Shatah_Strauss:solit_stability, Grillakis_Shatah_Strauss:solit_stability_2, LeCoz:standing_waves_NLS}). Indeed, it is known that the solitons $G_j$ are critical points and even ground states respectively for the action functionals
\begin{equation*}
    S_j (v) \coloneqq E(v) + \Bigl( 2 \lambda \omega_j + \frac{\abs{v_j}^2}{2} \Bigr) M (v) - v_j \cdot \mathcal{J} (v).
\end{equation*}
In the previous articles, the Hessian of these functionals is coercive on a subspace of $H^1$ of finite co-dimension in $L^2$. At large time, the components of the multi-soliton are well-separated and thus the analysis can be localized around each soliton to gain an $H^1$-local control, up to a space of finite dimension in $L^2$, for the linearized. Thus the sum of the localized functionals is (locally around $G$) $H^1$-coercive up to some negligible terms and few directions in $L^2$, which can be controlled either by modulation of the invariants (for ground states) or by a better control of the $L^2$-norm.

However, the fact of having few bad $L^2$ directions is not necessary, in particular since we already have an $L^2$-estimate thanks to Proposition \ref{prop:unif_est_L2}: the constructed functional $S^\textnormal{loc}$, which is time-dependent and "slowly varies" for all $u_n$ (in a sense explained later), only need to satisfy the fact that there exists some $K_1 > 0$ and $K_2 \in \mathbb{R}$ such that for all $t$ large enough and for all $v \in W$ (possibly near $G(t)$)
\begin{equation*}
    S (t,v) - \sum_j S_j (G_j) \geq K_1 \norm{w}_{H^1}^2 - K_2 \norm{w}_{L^2}^2,
\end{equation*}
with $w \coloneqq v - G(t)$.

This functional $S^\textnormal{loc}$ is constructed in a similar way as in the above articles.
It requires the use of an action-like functional, defined with quantities localized around each Gausson. For this, we need a localization procedure. However, unlike in \cite{Martel_Merle__Multi_solitary_waves_NLS, Cote_LeCoz_Multisolitons_NLS} for example, we will not localize only in one dimension by taking a particular direction in which all the speed components are well ordered. Indeed, such a tactics would lead to a slower convergence rate, since the minimal relative speed in this direction would be smaller.

Take some $T'' > T'$ and $t_* > 0$ to be fixed later and take any $n \in \mathbb{N}$ such that $T_n > T'' + t_*$. Set again $w_n \coloneqq u_n - G$. Let $t^\dag \in [t_*, T_n - T'']$ and assume that for all $t \in [t^\dag, T_n - T'']$ there holds
\begin{equation}
    \norm{\nabla w_n (t')}_{L^2} \leq e^{- \frac{\lambda (v_* t)^2}{4}}, \label{eq:ass_H1_est}
\end{equation}
with the notation $t' \coloneqq T'' + t$.
By Proposition \ref{prop:unif_est_L2}, we know that there also holds
\begin{equation}
    \norm{w_n (t')}_{L^2} \leq e^{- \frac{\lambda (v_* (t + \tau))^2}{4}}, \label{eq:L2_est}
\end{equation}
where $\tau \coloneqq T'' - T'$.

Take $\phi : \mathbb{R} \rightarrow \mathbb{R}$ a $\mathcal{C}^\infty$ function such that $\phi(s) = 1$ for all $s \leq -1$, $\phi(s) = 0$ for all $s \geq 1$ and $\phi (s) \in [0,1]$ and $-1 \leq \phi' (s) \leq 0$ for all $s \in \mathbb{R}$. Therefore we define:
\begin{itemize}
    \item the center of each Gausson
    \begin{equation*}
        x_j^* (t') \coloneqq x_j + t' \, v_j,
    \end{equation*}
    
    \item functions $\psi_j$ ($1 \leq j \leq N$)  with the $j$-th member weighted around the $j$-th Gausson:
    \begin{equation*}
        \psi_j (t',x) \coloneqq \phi \Bigl( \abs{x - x_j^* (t')} - \frac{v_* t}{2} - 2 \Bigr) \qquad \text{ for } j = 1, \dots, N.
    \end{equation*}
    
    \item A last function $\psi_0$  completing the previous family so that $(\psi_j)_{0 \leq j \leq N}$ is a partition of unity:
    \begin{equation*}
        \psi_0 \coloneqq 1 - \sum_{j=1}^N \psi_j.
    \end{equation*}
\end{itemize}

\begin{figure}[htb!]
  \centering
  \begin{tikzpicture}[scale=0.8]
    \tkzInit[xmin=-8,xmax=8,ymin=-8,ymax=8]
    \clip (-8,8) rectangle (8,-8);
    \fill[pattern=horizontal lines,opacity=0.2, thick] (-8,8) rectangle (8,-8);
    
    \node (x1) at (-4,0) [] {$\times$};
    \node (x2) at (3,3.5) [] {$\times$};
    \node (x3) at (2.5,-4) [] {$\times$};
    
    \draw[dashed, fill=white] (x1) circle (1.5) node[below=0.5] {$\psi_1 \equiv 1$};
    \draw[dashed, fill=white] (x2) circle (1.5) node[below=0.5] {$\psi_2 \equiv 1$};
    \draw[dashed, fill=white] (x3) circle (1.5) node[above=0.5] {$\psi_3 \equiv 1$};
    
    \node (x1) at (-4,0) [] {$\times$};
    \node (x2) at (3,3.5) [] {$\times$};
    \node (x3) at (2.5,-4) [] {$\times$};
    
    \draw (x1) node[above right] {$x_1 (t)$};
    \draw (x2) node[above left] {$x_2 (t)$};
    \draw (x3) node[left] {$x_3 (t)$};
    
    \draw [->, very thick] (-4,0) -- (-6,0.5) node[above] {$v_1$};
    \draw [->, very thick] (3,3.5) -- (3.5,4) node[above] {$v_2$};
    \draw [->, very thick] (2.5,-4) -- (2.7,-5) node[left] {$v_3$};
    
    \draw[thick] (x1) circle (2.5) node[below=1.3] {$\supp \psi_1$};
    \draw[thick] (x2) circle (2.5) node[below=1.3] {$\supp \psi_2$};
    \draw[thick] (x3) circle (2.5) node[below=1.3] {$\supp \psi_3$};
    
    \draw [->] (-6.5,0) -- (-7,0);
    \draw [->] (-4,2.5) -- (-4,3);
    \draw [->] (-1.5,0) -- (-1,0);
    \draw [->] (-4,-2.5) -- (-4,-3);
    
    \draw [->] (0.5,3.5) -- (0,3.5);
    \draw [->] (3,6) -- (3,6.5);
    \draw [->] (5.5,3.5) -- (6,3.5);
    \draw [->] (3,1) -- (3,0.5);
    
    \draw [->] (0,-4) -- (-0.5,-4);
    \draw [->] (2.5,-1.5) -- (2.5,-1);
    \draw [->] (5,-4) -- (5.5,-4);
    \draw [->] (2.5,-6.5) -- (2.5,-7);

  \end{tikzpicture}
  \caption{Schematic representation for the partition of unity
    $(\psi_j)$ (in dimension 2). The support of $\psi_0$ is given by the horizontal lines.}
  \label{fig:partition}
\end{figure}
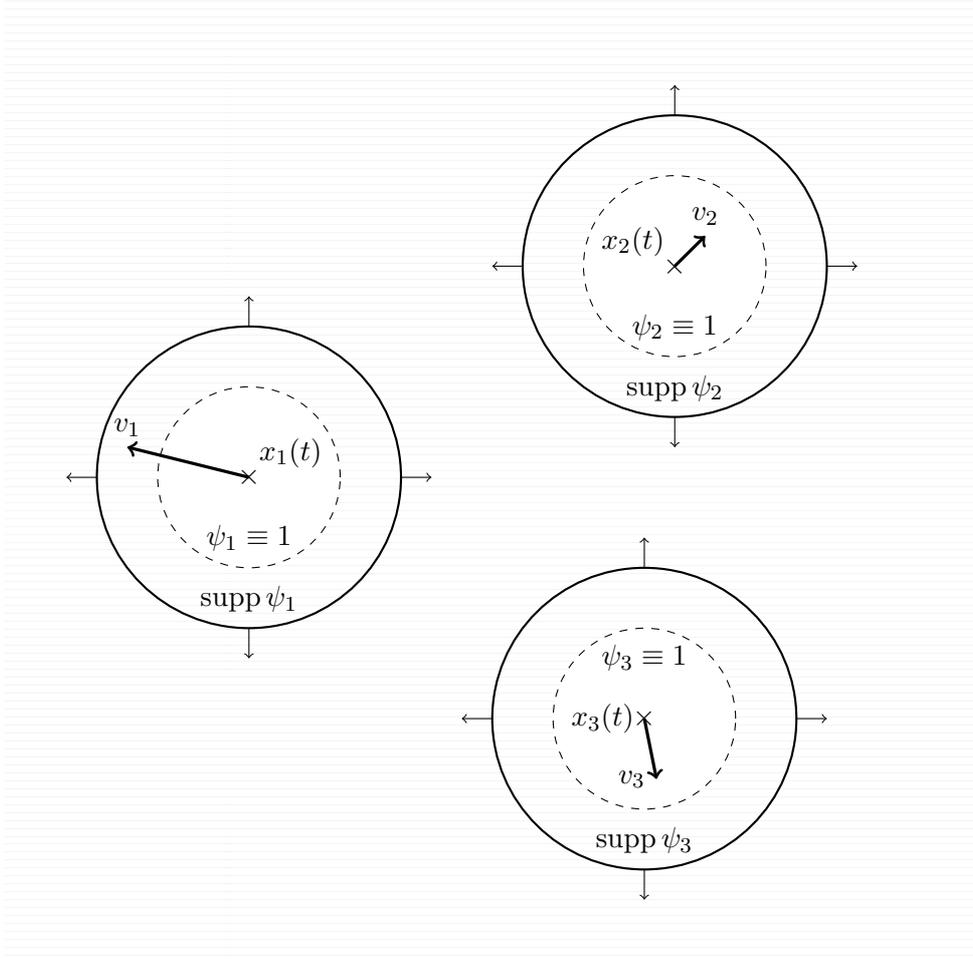

With this definition, $(\psi_j (t'))_j$ is a (time-dependent) smooth partition of unity in space for $t \geq 0$ as soon as $T''$ is large enough (see Figure \ref{fig:partition} for a schematic representation).
We can now define \emph{localized quantities} which will turn out to be almost conserved: for $j = 0, \dots, N$ and $t \geq 0$, set
\begin{gather*}
    M_j (t', v) \coloneqq \int |v|^2 \, \psi_j (t') \diff x, \qquad \mathcal{J}_j (t',v) \coloneqq \Im \int \nabla v \, \overline{v} \, \psi_j (t') \diff x, \\
    E_j (t', v) \coloneqq \frac{1}{2} \int \abs{\nabla v}^2 \, \psi_j (t') \diff x - \lambda \int \abs{v}^2 (\ln \abs{v}^2 - 1) \, \psi_j (t') \diff x.
\end{gather*}
We know that each $G_j$ is well fitted for the action (for $j=0$, take $G_0 = G^d$, $\omega_0 = 0$ and $v_0 = 0$)
\begin{equation*}
    S_j (v) \coloneqq E (v) + \Bigl( 2 \lambda \omega_j + \frac{\abs{v_j}^2}{2} \Bigr) M (v) - v_j \cdot \mathcal{J} (v).
\end{equation*}
Then, we also define as well localized actions (for $j = 0, \dots, N$) by:
\begin{equation*}
    S_j^\textnormal{loc} (t',v) \coloneqq E_j (t',v) + \Bigl( 2 \lambda \omega_j + \frac{\abs{v_j}^2}{2} \Bigr) M_j (t',v) - v_j \cdot \mathcal{J}_j (t',v).
\end{equation*}
Finally, we define a localized action-like functional for multi-solitons:
\begin{equation*}
    S^\textnormal{loc} (t',v) \coloneqq \sum_j S_j^\textnormal{loc} (t',v) = E (v) + \sum_{j \geq 1} \Bigl( 2 \lambda \omega_j + \frac{\abs{v_j}^2}{2} \Bigr) M_j (t',v) - v_j \cdot \mathcal{J}_j (t',v).
\end{equation*}

Our aim is to prove that $S^\textnormal{loc} (t', v)$ is almost coercive around $G(t')$, but also that $S^\textnormal{loc} (t', u_n (t'))$ slowly varies. To be more precise, we will prove the two following propositions.
\begin{prop}[Almost-coercivity] \label{prop:coercivity}
    If $t_*$ is large enough, then for all $t \in [t^\dag, T_n - T'']$,
    \begin{equation*}
        S^\textnormal{loc} (t',u_n(t')) - \sum_{j \geq 1} S_j (G_j) \geq \frac{1}{2} \norm{w_n (t')}_{\dot{H}^1}^2 + C_0 \, t^{-1} e^{- \frac{\lambda (v_* t)^2}{2}}.
    \end{equation*}
\end{prop}

\begin{prop}[Slow variations] \label{prop:slow_var_loc_func}
    For all $t \in [t^\dag, T_n - T'']$, there holds
    \begin{equation*}
        S^\textnormal{loc} (t',u_n(t')) - \sum_{j \geq 1} S_j (G_j) \leq C_0 \, t^{-1} e^{- \frac{\lambda (v_* t)^2}{2}}.
    \end{equation*}
\end{prop}

Before proving these two propositions, we show how they lead to the Bootstrap Property \ref{prop:bootstrap}.

\begin{proof}[Proof of Proposition \ref{prop:bootstrap}]
    By \eqref{eq:L2_est} and Proposition \ref{prop:coercivity}, for all $t \in [t^\dag, T_n - T'']$
    \begin{equation*}
        \norm{w_n (t')}_{\dot{H}^1}^2 \leq 2 \Bigl( S^\textnormal{loc} (t',u_n(t')) - \sum_j S_j (G_j) \Bigr) + C_0 \, t^{-1} e^{- \frac{\lambda (v_* t)^2}{2}} .
    \end{equation*}
    Moreover, Proposition \ref{prop:slow_var_loc_func} yields
    \begin{equation*}
        S^\textnormal{loc} (t',u_n(t')) - \sum_j S_j (G_j) \leq C_0 \, t^{-1} e^{- \frac{\lambda (v_* t)^2}{2}}.
    \end{equation*}
    Thus, combining this with the previous inequality, we get as soon as $t_*$ is large enough:
    \begin{equation*}
        \norm{w_n (t')}_{H^1}^2 \leq \frac{1}{2} e^{- \frac{\lambda (v_* t)^2}{2}}, \qquad \forall t \in [t^\dag, T_n - T''],
    \end{equation*}
    which gives the conclusion.
\end{proof}

\subsection{Almost-coercivity of the functional}

We now prove Proposition \ref{prop:coercivity}, which requires to linearize the functional $S^\textnormal{loc}$ in terms of $w_n$ and $G$. However, $E$ is not $\mathcal{C}^2$ because of its potential energy, and we cannot linearize exactly like in \cite{Martel_Merle__Multi_solitary_waves_NLS,Cote_LeCoz_Multisolitons_NLS} for example. However, a weaker expansion still holds in this sense:
\begin{lem} \label{lem:linearize_action}
    For all $t \geq 0$, $n \in \mathbb{N}$ and $j \in \{ 1, \dots, N \}$, there holds
    \begin{multline}
        S_j^\textnormal{loc} (t',u_n(t')) \geq S_j^\textnormal{loc} (t', G_j(t')) + H_j (t', w_n^j (t')) - \int \Re \Bigl( \nabla G_j (t') \, \overline{w_n^j (t')} \Bigr) \cdot \nabla \psi_j (t') \diff x \\ + v_j \cdot \Im \int \overline{G_j (t')} \, w_n^j (t') \, \nabla \psi_j (t') \diff x, \label{eq:linearize_action}
    \end{multline}
    where $w_n^j \coloneqq u_n - G_j$ and
    \begin{multline}
        H_j (t', w) \coloneqq \frac{1}{2} \int \abs{\nabla w}^2 \, \psi_j (t') \diff x - 2 \lambda \int \abs{w}^2 \, \biggl( \ln{\Bigl( 1 + \abs{w} \Bigr)} + C_0 \biggr) \, \psi_j (t') \diff x \\ + \Bigl( 2 \lambda \omega_j + \frac{\abs{v_j}^2}{2} \Bigr) M_j (t',w) - v_j \cdot \mathcal{J}_j (t',w). \label{eq:linearized}
    \end{multline}
\end{lem}

This Lemma is a corollary of a kind of weak expansion of the function $F_1 (z) = \abs{z} \ln{\abs{z}^2}$ (see Lemma \ref{lem:exp_u_ln_u}).
The second term in \eqref{eq:linearized} is not what one would expect in order to be able to reproduce the proof of \cite{Martel_Merle__Multi_solitary_waves_NLS}, since it is not the linearized potential energy. However, since we have already obtained the $L^2$ uniform estimates for $w_n$, it is still enough in order to prove Proposition \ref{prop:coercivity}. Before proving this lemma, we show how it is used to prove Proposition \ref{prop:coercivity}.

\subsubsection{Proof of Proposition \ref{prop:coercivity}}

First of all, we will need some results about the computation of some quantities about $G_j$ outside its "physical support".

\begin{lem} \label{lem:est_H1_gauss}
    For all $j \in \{ 1, \dots, N \}$ and $k \in \{ 0, \dots, N \}$, there holds when $t \rightarrow \infty$
    \begin{gather*}
        \norm{G_j (t')}_{L^2 ( (1-\psi_j (t')) \diff x )} + \norm{G_j (t')}_{L^2 (\norm{D_{xxx}^3 \psi_k (t')} \diff x)} + \norm{ G_j (t')}_{L^2 ( \abs{\nabla \psi_k (t')} \diff x )} = o \biggl( t^{-3} e^{- \frac{\lambda (v_* t)^2}{4}} \biggr), \\
        \norm{\nabla G_j (t')}_{L^2 ( (1-\psi_j (t')) \diff x )} + \norm{\abs{\nabla \psi_k (t')} \, \nabla G_j (t')}_{L^2} = o \biggl( t^{-1} e^{- \frac{\lambda (v_* t)^2}{4}} \biggr), \\
        \norm{G_j (t')}_{L^{2+\frac{1}{d}} ( (1-\psi_j (t')) \diff x )} = o \biggl( t^{-3} e^{- \frac{\lambda (v_* t)^2}{4}} \biggr), \\
        \norm{G_j (t') \abs{x - x_j^* (t')}^2}_{L^2 ( (1-\psi_j (t')) \diff x )} + \norm{G_j (t') \abs{x - x_j^* (t')}^3}_{L^2 ( (1-\psi_j (t')) \diff x )} = o \biggl( t^{-1} e^{- \frac{\lambda (v_* t)^2}{4}} \biggr).
    \end{gather*}
    %
\end{lem}

The lemma follows from the support properties of $\psi_k$ and the Gaussian decay of $G_j$.
We postpone its proof to Appendix \ref{sec_app:proof_lem_gauss}.
In particular, since $0 \leq \psi_k \leq 1 - \psi_j$ for all $k \neq j$, this gives a simple corollary:

\begin{cor} \label{cor:est_H1_gauss}
    For all $k \geq 0$, $j \geq 1$ such that $k \neq j$, there holds when $t \rightarrow \infty$
    \begin{gather*}
        \norm{G_j (t')}_{L^2 (\psi_k (t') \diff x)} + \norm{G_j (t')}_{L^{2+\frac{1}{d}} ( \psi_k (t') \diff x )} = o \biggl( t^{-3} e^{- \frac{\lambda (v_* t)^2}{4}} \biggr), \\
        \norm{\psi_k (t') \, \nabla G_j (t')}_{L^2} + \norm{\nabla G_j (t')}_{L^2 ( \psi_k (t') \diff x )} = o \biggl( t^{-1} e^{- \frac{\lambda (v_* t)^2}{4}} \biggr), \\
        \norm{G_j (t') \abs{x - x_j^* (t')}^3}_{L^2 ( \psi_k (t') \diff x )} = o \biggl( t^{-1} e^{- \frac{\lambda (v_* t)^2}{4}} \biggr).
    \end{gather*}
\end{cor}

Moreover, we also know that $\ln{\abs{G_j (t')}^2} = 2 \omega_j + d - 2 \lambda \abs*[\normal]{x - x_j^* (t')}^2$, so we can also derive the same decay estimate for $G_j(t') \ln{\abs{G_j (t')}^2}$:

\begin{cor} \label{cor:est_L_log_L_gauss}
    For all $k \geq 0$, $j \geq 1$ such that $k \neq j$, there holds when $t \rightarrow \infty$
    \begin{gather*}
        \norm{G_j (t') \ln{\abs{G_j (t')}^2}}_{L^2 ((1 - \psi_j (t')) \diff x)} = o \biggl( t^{-1} e^{- \frac{\lambda (v_* t)^2}{4}} \biggr), \\
        \norm{G_j (t') \ln{\abs{G_j (t')}^2}}_{L^2 (\psi_k (t') \diff x)} = o \biggl( t^{-1} e^{- \frac{\lambda (v_* t)^2}{4}} \biggr).
    \end{gather*}
\end{cor}

Since the Gaussons get away from each other and with their Gaussian decay, we also show that they are almost orthogonal:
\begin{lem} \label{lem:orth_gauss}
    For all $j \neq k \in \{ 1, \dots, N \}$ and $\ell \geq 1$, there holds
    \begin{multline*}
        \int \abs{G_j (t')} \, \abs{G_k (t')} (1 + \abs{x-x_\ell^* (t')}^2) \diff x +
        \int \abs{\nabla G_j (t')} \, \abs{G_k (t')} \diff x \\ +
        \int \abs{\nabla G_j (t')} \, \abs{\nabla G_k (t')} \diff x = o \biggl( t^{-1} e^{- \frac{\lambda (v_* t)^2}{2}} \biggr).
    \end{multline*}
\end{lem}

\begin{proof}
    For the third term, we have
    \begin{align*}
        \int \abs{\nabla G_j (t')} \, \abs{\nabla G_k (t')} \diff x &= \int \abs{ i v_j - 2 \lambda (x - x_j^* (t')) } \abs{ i v_k - 2 \lambda (x - x_k^* (t')) } \abs{G_j (t',x)} \abs{G_k (t',x)} \diff x \\
        &\begin{multlined}[t][10cm] = e^{\omega_j + \omega_k} \int \abs{ i v_j - 2 \lambda \biggl(y + \frac{x_k^* (t') - x_j^* (t')}{2} \biggr) } \abs{ i v_k - 2 \lambda \biggl(y - \frac{x_k^* (t') - x_j^* (t')}{2} \biggr) } \\ \exp \Bigl[ - 2 \lambda \abs{y}^2 - \lambda \frac{(x_k^* (t') - x_j^* (t'))^2}{2} \Bigr] \diff y \end{multlined} \\
        &\begin{multlined}[t][11cm] \leq C_0 \, \exp \Bigl[ - \lambda \frac{(x_k^* (t') - x_j^* (t'))^2}{2} \Bigr] \int \Bigl( 1 + \abs{y}^2 + \abs{x_k^* (t') - x_j^* (t')}^2 \Bigr) \exp \Bigl[ - 2 \lambda \abs{y}^2 \Bigr] \diff y \end{multlined} \\
        &\leq C_0 \, \exp \Bigl[ - \lambda \frac{\abs*[\big]{x_k^* (t') - x_j^* (t')}^2}{2} \Bigr] \Bigr( \abs{x_k^* (t') - x_j^* (t')}^2 + 1 \Bigl) \\
        &= o \biggl( t^{-1} e^{- \frac{\lambda (v_* t)^2}{2}} \biggr),
    \end{align*}
    since $\abs{x_k^* (t') - x_j^* (t')} \geq \varepsilon_0^{-1} + v_* (t + \tau)$ for all $j < k$ thanks to \eqref{eq:distance_x_k}.
    The same kind of computation can be performed for the other terms.
\end{proof}

Now, we have all the tools that we need in order to prove Proposition \ref{prop:coercivity} from Lemma \ref{lem:linearize_action}. First, we show that the last two terms in \eqref{eq:linearize_action} are negligible.

\begin{cor} \label{cor:neg_terms}
    For all $n \in \mathbb{N}$, $t \in [t^\dag, T_n - T'']$, $j \in \{ 1, \dots, N \}$,
    \begin{equation*}
        \abs{\int \Re \Bigl( \nabla G_j (t') \, \overline{w_n^j (t')} \Bigr) \, \nabla \psi_j (t) \diff x} + \abs{\Im \int \overline{G_j (t')} \, w_n^j (t') \, \nabla \psi_j (t') \diff x} \leq C_0 \, t^{-1} e^{- \frac{\lambda (v_* t)^2}{2}}.
    \end{equation*}
\end{cor}

\begin{proof}
    We recall that
    \begin{equation*}
        w_n^j = w_n + \sum_{k\neq j} G_k,
    \end{equation*}
    so that
    \begin{equation*}
        \abs{\int \Re \Bigl( G_j (t') \, \overline{w_n^j (t')} \Bigr) \, \nabla \psi_j (t) \diff x} \leq \abs{\int \Re \Bigl( G_j (t') \, \overline{w_n (t')} \Bigr) \, \nabla \psi_j (t) \diff x} + \sum_{k \neq j} \abs{\int \Re \Bigl( G_j (t') \, \overline{G_k (t')} \Bigr) \, \nabla \psi_j (t) \diff x}.
    \end{equation*}
    By \eqref{eq:L2_est} and Lemma \ref{lem:est_H1_gauss}, there holds for all $t \in [t^\dag, T_n - T'']$
    \begin{equation*}
        \abs{\int \Re \Bigl( G_j (t') \, \overline{w_n (t')} \Bigr) \, \nabla \psi_j (t) \diff x} \leq \norm{w_n (t')}_{L^2} \norm{\nabla \psi_k (t') \, G_j (t')}_{L^2} \leq C_0 \, t^{-1} e^{- \frac{\lambda (v_* t)^2}{2}}.
    \end{equation*}
    Moreover, Lemma \ref{lem:orth_gauss} gives
    \begin{equation*}
        \abs{\int \Re \Bigl( G_j (t') \, \overline{G_k (t')} \Bigr) \, \nabla \psi_j (t) \diff x} \leq \int \abs{G_j (t')} \, \abs{G_k (t')} \diff x \leq C_0 \, t^{-1} e^{- \frac{\lambda (v_* t)^2}{2}},
    \end{equation*}
    which gives the conclusion for the first term. A similar computation holds for the second term.
\end{proof}

Then, we can also substitute $S_j^\textnormal{loc} (t', G_j(t'))$ by $S_j (G_j)$ up to a negligible term.

\begin{cor} \label{cor:diff_actions_for_gauss}
    For all $j \in \{ 1, \dots, N \}$, for all $t \geq 0$
    \begin{equation*}
        \abs{S_j^\textnormal{loc} (t', G_j(t')) - S_j (G_j)} \leq C_0 \, t^{-1} e^{- \frac{\lambda (v_* t)^2}{2}}.
    \end{equation*}
\end{cor}

\begin{proof}
    With a simple computation,
    \begin{multline*}
        S_j (G_j) - S_j^\textnormal{loc} (t', G_j(t')) = \frac{1}{2} \int \abs{\nabla G_j (t')}^2 \, (1-\psi_j (t')) \diff x - \lambda \int \abs{G_j (t')}^2 (\ln \abs{G_j (t')}^2 - 1) \, (1-\psi_j (t')) \diff x \\
            + \biggl( 2 \lambda \omega_j + \frac{\abs{v_j}^2}{2} \biggr) \int |G_j|^2 (1-\psi_j (t')) \diff x - v_j \cdot \Im \int \nabla G_j \, \overline{G_j} \, (1-\psi_j (t')) \diff x.
    \end{multline*}
    The conclusion readily follows from Lemma \ref{lem:est_H1_gauss} and Corollary \ref{cor:est_L_log_L_gauss}.
\end{proof}

An important feature of $H = \sum_j H_j$ is that it is coercive in $H^1$, up to an $L^2$ norm. Since we already know that the $L^2$ norm of $w_n$ is negligible, its $H^1$ norm is therefore controlled by $H$. In order to prove this, we have this first result about some coercivity of $H_j$:

\begin{lem} \label{lem:coerc_lin}
    For all $n \in \mathbb{N}$, $t \in [t^\dag, T_n - T'']$, $j \in \{ 1, \dots, N \}$,
    \begin{equation*}
        H_j (t', w_n^j (t')) \geq \frac{1}{2} \int \abs{\nabla w_n}^2 \, \psi_j (t') \diff x - C_0 \, t^{-1} e^{- \frac{\lambda (v_* t)^2}{2}}.
    \end{equation*}
\end{lem}

\begin{proof}
    First, we have
    \begin{multline*}
        H_j (t', w_n^j (t')) \coloneqq \frac{1}{2} \int \abs{\nabla w_n^j (t')}^2 \, \psi_j (t') \diff x - 2 \lambda \int \abs{w_n^j (t')}^2 \, \biggl( \ln{\Bigl( 1 + \abs{w_n^j (t')} \Bigr)} + C_0 \biggr) \, \psi_j (t') \diff x \\ + \Bigl( 2 \lambda \omega_j + \frac{\abs{v_j}^2}{2} \Bigr) M_j (t',w_n^j (t')) - v_j \cdot \mathcal{J}_j (t',w_n^j (t')).
    \end{multline*}
    We also recall that $w_n^j = w_n + \sum_{k \neq j} G_k$. 
    
    \begin{itemize}
        \item For the first term, we have
        \begin{multline*}
            \int \abs{\nabla w_n^j (t')}^2 \, \psi_j (t') \diff x = \int \abs{\nabla w_n (t')}^2 \, \psi_j (t') \diff x + 2 \sum_{k \neq j} \Re \int \overline{\nabla w_n (t')} \, \nabla G_k (t') \, \psi_j (t') \diff x \\ + \int \abs{\sum_{k \neq j} \nabla G_k (t')}^2 \, \psi_j (t') \diff x.
        \end{multline*}
        For the second term in the right-hand side, we compute for any $k \neq j$ thanks to \eqref{eq:ass_H1_est} and Corollary \ref{cor:est_H1_gauss}:
        \begin{equation*}
            \abs{\Re \int \overline{\nabla w_n (t')} \, \nabla G_k (t') \, \psi_j (t') \diff x} \leq \norm{\nabla w_n (t')}_{L^2} \norm{\psi_j (t') \, \nabla G_k (t')}_{L^2} \leq C_0 \, t^{-1} e^{- \frac{\lambda (v_* t)^2}{2}}.
        \end{equation*}
        As for the third term, using Corollary \ref{cor:est_H1_gauss}, we compute in the same way:
        \begin{equation*}
            \norm{\sum_{k \neq j} \nabla G_k (t')}_{L^2 (\psi_j (t') \diff x)} \leq \sum_{k \neq j} \norm{\nabla G_k (t')}_{L^2 (\psi_j (t') \diff x)} \leq C_0 \, t^{-1} e^{- \frac{\lambda (v_* t)^2}{4}}.
        \end{equation*}

        \item For the second term, we use the fact that we have $C_d > 0$ such that for all $z \in \mathbb{R}_+$,
        \begin{equation*}
            z^2 \, \biggl( \ln{\Bigl( 1 + z \Bigr)} + C_0 \biggr) \leq C_d ( z^2 + z^{2+\frac{1}{d}} ),
        \end{equation*}
        so that
        \begin{equation*}
            \abs{\int \abs{w_n^j (t')}^2 \, \biggl( \ln{\Bigl( 1 + \abs{w_n^j (t')} \Bigr)} + C_0 \biggr) \, \psi_j (t') \diff x} \leq C_d \Bigl( \norm{w_n^j (t')}_{L^2 (\psi_j (t') \diff x)}^2 + \norm{w_n^j (t')}_{L^{2+\frac{1}{d}} (\psi_j (t') \diff x)}^{2+\frac{1}{d}} \Bigr).
        \end{equation*}
        Then, using Lemma \ref{lem:est_H1_gauss} and Corollary \ref{cor:est_H1_gauss}, we have
        \begin{equation*}
            \norm{w_n^j (t')}_{L^2 (\psi_j (t') \diff x)} \leq \norm{w_n (t')}_{L^2} + \sum_{k \neq j} \norm{G_k (t')}_{L^2 (\psi_j (t') \diff x)} \leq C_0 \, t^{-1} e^{- \frac{\lambda (v_* t)^2}{4}},
        \end{equation*}
        and
        \begin{align*}
            \norm{w_n^j (t')}_{L^{2+\frac{1}{d}} (\psi_j (t') \diff x)} &\leq \norm{w_n (t')}_{L^{2+\frac{1}{d}}} + \sum_{k \neq j} \norm{G_k (t')}_{L^{2+\frac{1}{d}} (\psi_j (t') \diff x)} \\
                &\leq C_d \, \norm{w_n (t')}_{H^1} + C_0 \, e^{- \frac{\lambda (v_* t)^2}{4}} \leq C_0 \, e^{- \frac{\lambda (v_* t)^2}{4}}.
        \end{align*}
        Thus,
        \begin{equation*}
            \abs{\int \abs{w_n^j (t')}^2 \, \biggl( \ln{\Bigl( 1 + \abs{w_n^j (t')} \Bigr)} + C_0 \biggr) \, \psi_j (t') \diff x} \leq C_0 \, t^{-1} e^{- \frac{\lambda (v_* t)^2}{2}}.
        \end{equation*}

        \item The last two terms can be easily estimated in the same way since $M_j (t',w_n^j (t')) = \norm{w_n^j}_{L^2 (\psi_j (t') \diff x)}^2$ and
        \begin{equation*}
            \abs{\mathcal{J}_j (t',w_n^j (t'))} \leq \norm{w_n^j}_{L^2 (\psi_j (t') \diff x)} \norm{\nabla w_n^j}_{L^2 (\psi_j (t') \diff x)}. \qedhere
        \end{equation*}
    \end{itemize}
\end{proof}

By putting Corollaries \ref{cor:neg_terms} and \ref{cor:diff_actions_for_gauss} and Lemma \ref{lem:coerc_lin} in Lemma \ref{lem:linearize_action}, we easily deduce a nice "coercivity" property for the localized functionals $S_j^\textnormal{loc}$ for $j \geq 1$:

\begin{cor} \label{cor:coerc_loc_act}
    For all $n \in \mathbb{N}$, $t \in [t^\dag, T_n - T'']$, $j \in \{ 1, \dots, N \}$,
    \begin{equation*}
        S_j^\textnormal{loc} (t',u_n(t')) - S_j (G_j) \geq \frac{1}{2} \int \abs{\nabla w_n}^2 \, \psi_j (t') \diff x - C_0 \, t^{-1} e^{- \frac{\lambda (v_* t)^2}{2}}.
    \end{equation*}
\end{cor}

As for the case $j=0$, a similar property holds:

\begin{lem} \label{lem:coerc_loc_act0}
    For all $n \in \mathbb{N}$, $t \in [t^\dag, T_n - T'']$
    \begin{equation*}
        S_0^\textnormal{loc} (t',u_n(t')) \geq \frac{1}{2} \int \abs{\nabla w_n}^2 \, \psi_0 (t') \diff x - C_0 \, t^{-1} e^{- \frac{\lambda (v_* t)^2}{2}}.
    \end{equation*}
\end{lem}{}

\begin{proof}
    \begin{equation*}
        S_0^\textnormal{loc} (t',u_n(t')) = \frac{1}{2} \int \abs{\nabla u_n (t')}^2 \, \psi_0 (t') \diff x - \lambda \int \abs{u_n (t')}^2 (\ln \abs{u_n (t')}^2 - 1) \, \psi_0 (t') \diff x.
    \end{equation*}{}

    \begin{itemize}
        \item For the first term:
        \begin{equation*}
            \int \abs{\nabla u_n}^2 \, \psi_0 (t') \diff x = \int \abs{\nabla w_n}^2 \, \psi_0 (t') \diff x + 2 \int \Re{\nabla w_n \cdot \overline{\nabla G}} \, \psi_0 (t') \diff x + \int \abs{\nabla G }^2 \, \psi_0 (t') \diff x.
        \end{equation*}{}
        The last two terms of the right-hand side can be easily estimated. First:
        \begin{align*}
            \abs{\int \Re{\nabla w_n \cdot \overline{\nabla G}} \, \psi_0 (t') \diff x} &\leq \norm{\nabla w_n}_{L^2} \norm{\psi_0 (t') \, \nabla G (t')}_{L^2} \\
            &\leq \norm{\nabla w_n}_{L^2} \sum_{j \geq 1} \norm{\psi_0 (t') \, \nabla G_j (t')}_{L^2} \leq C_0 \, t^{-1} e^{- \frac{\lambda (v_* t)^2}{2}}.
        \end{align*}{}
        Then we use Corollary \ref{cor:est_H1_gauss} for the last term.
        
        \item For the second term, we show it is negligible by using the fact that for all $y > e$:
        \begin{equation*}
            y \, (\ln y - 1) \leq C_d \, y^{1+\frac{1}{2 d}}
        \end{equation*}{}
        for some $C_d > 0$.
        Thus,
        \begin{align*}
            \int \abs{u_n}^2 (\ln \abs{u_n}^2 - 1) \, \psi_0 (t) \diff x &\leq \int_{\abs{u_n}^2 > e} \abs{u_n}^2 (\ln \abs{u_n}^2 - 1) \, \psi_0 (t) \diff x \\
                &\leq C_d \, \norm{u_n (t')}_{L^{2+\frac{1}{d}} (\psi_0 (t') \diff x)}^{2+\frac{1}{d}}.
        \end{align*}
        Then, the conclusion comes from:
        \begin{align*}
            \norm{u_n (t')}_{L^{2+\frac{1}{d}} (\psi_0 (t') \diff x)} &\leq \norm{w_n (t')}_{L^{2+\frac{1}{d}}} + \sum_{j \geq 1} \norm{G_j (t')}_{L^{2+\frac{1}{d}} (\psi_0 (t') \diff x)} \\
                &\leq C_d \, \norm{w_n (t')}_{H^1} + C_0 \, e^{- \frac{\lambda (v_* t)^2}{4}} \leq C_0 \, e^{- \frac{\lambda (v_* t)^2}{4}}. \qedhere
        \end{align*}
    \end{itemize}{}
\end{proof}

Proposition \ref{prop:coercivity} is then a simple corollary from these results, by summing over $j$ Corollary \ref{cor:coerc_loc_act} and Lemma \ref{lem:coerc_loc_act0}.

\subsubsection{Proof of Lemma \ref{lem:linearize_action}}

This lemma mostly relies on an inequality for the potential energy with an expression near the expected expansion. To prove this, set
\begin{align*}
    F_1 : \ &\mathbb{C} \rightarrow \mathbb{R}
     \\ &z \mapsto \abs{z}^2 \, (\ln \abs{z}^2 - 1),
\end{align*}
so that the potential energy is $- \lambda \int F_1 (v)$.
Then, the following inequality holds:

\begin{lem} \label{lem:exp_u_ln_u}
    For all $z_1, z_2 \in \mathbb{C}$, set $\zeta \coloneqq z_1 - z_2$. Then
    \begin{equation*}
        F_1 (z_1) \leq F_1 (z_2) + 2 \Re \left( z_2 \overline{\zeta} \right) \, \ln \abs{z_2}^2 + 2 \, \abs{\zeta}^2 \, \biggl( \ln{\Bigl( \max( \abs{z_2}, \abs{z_1} ) \Bigr)} + 1 \biggr).
    \end{equation*}
\end{lem}
\begin{rem} \label{rem:last_term}
    In the case $z_2 = 0$, the second term of the right-hand side is to be understood as being $0$.
    If furthermore $z_1 = 0$, then so is the last term of the right-hand side.
\end{rem}

\begin{rem}
    Like already pointed out, the third term of the right-hand side is not what one would expect in order to be able to reproduce the proof of \cite{Martel_Merle__Multi_solitary_waves_NLS,Cote_LeCoz_Multisolitons_NLS} for instance. Indeed, the expected formula would be something of the form:
    \begin{equation*}
        F_1 (z_1) = F_1 (z_2) + 2 \Re \left( z_2 \overline{\zeta} \right) \, \ln \abs{z_2}^2 + \abs{\zeta}^2 \ln{\abs{z_2}^2} + 2 \frac{1}{\abs{z_2}^2} \Bigl( \Re \left( z_2 \overline{\zeta} \right) \Bigr)^2 + h(\zeta, z_2),
    \end{equation*}
    where $h$ is (at least) bounded when $\zeta$ and $z_2$ are bounded (and presumably of order more than 2 in $\zeta$).
    However, taking $z_1 = 1$ and $z_2 \rightarrow 0$ gives a simple counter-example for such an expansion.

    Moreover, if one takes $z_1 = u (x)$ for some $u \in W$ and $z_2 = e^{-\abs{x}^2}$ for instance and integrate, it gives:
    \begin{multline*}
        \int F_1 (u (x)) = \int F_1 (e^{-\abs{x}^2}) - 4 \int \Re \left( e^{-\abs{x}^2} \overline{\zeta (x)} \right) \, \abs{x}^2 - 2 \int \abs{\zeta (x)}^2 \abs{x}^2 + 2 \int \Bigl( \Re \left( \zeta (x) \right) \Bigr)^2 \\ + \int h(\zeta (x), e^{-\abs{x}^2}),
    \end{multline*}
    where one would want the last term to be controlled by the $W$-norm of $\zeta (x)$. With so, every term is bounded expect the third term of the right-hand side, which could be $- \infty$ if we take $u \notin \mathcal{F} (H^1)$ (such a $u$ exists).
\end{rem}

\begin{proof}
    For this proof only, we use the identification $\mathbb{C} \approx \mathbb{R}^2$, and we see $F_1$ as a function from $\mathbb{R}^2$ into $\mathbb{R}$ :
    \begin{align*}
        F_1 : \ &\mathbb{C} \approx \mathbb{R}^2 \rightarrow \mathbb{R}
         \\ &z = \begin{bmatrix}
            z_r \\
            z_i
            \end{bmatrix} \in \mathbb{R}^2 \mapsto \abs{z}^2 \, (\ln \abs{z}^2 - 1).
    \end{align*}
    For $z=0$, $F_1 (0) = 0$.
    Then, $F$ is differentiable on $\mathbb{C}$ and twice differentiable on $\mathbb{C} \setminus \{ 0 \}$ and we can compute for $z \neq 0$:
    \begin{equation*}
        \nabla F_1 (z) = \begin{bmatrix}
            2 z_r \ln \abs{z}^2 \\
            2 z_i \ln \abs{z}^2
            \end{bmatrix} = 2 z \ln \abs{z}^2.
    \end{equation*}
    We also compute $\nabla F(0) = 0$. Then, for all $z \neq 0$, we can differentiate again:
    \begin{equation*}
        D^2 F_1 (z) = 2 \begin{bmatrix}
            \ln \abs{z}^2 + 2 \frac{z_r^2}{\abs{z}^2} & 2 \frac{z_r z_i}{\abs{z}^2} \\
             2\frac{z_r z_i}{\abs{z}^2} & \ln \abs{z}^2 + 2 \frac{z_i^2}{\abs{z}^2}
            \end{bmatrix} = 2 \, R_z^{-1} L(z) R_z,
    \end{equation*}
    where $R_z$ is the rotation which maps $z$ onto the real positive half-line of $\mathbb{C}$ and $L (z)$ is defined by:
    \begin{equation*}
        L (z) = \begin{bmatrix}
            \ln \abs{z}^2 + 2 & 0 \\
            0 & \ln \abs{z}^2
            \end{bmatrix}.
    \end{equation*}
    In particular, we see that for all $z\neq0$ and all $h \in \mathbb{R}^2$, there holds
    \begin{equation*}
        \langle h, D^2 F_1 (z) \, h \rangle = 2 \langle R_z h, L (z) R_z h \rangle \leq 2 (\ln \abs{z}^2 + 2) \, \abs{R_z h}^2 = 4 (\ln \abs{z} + 1) \, \abs{h}^2.
    \end{equation*}
    Take $z_1, z_2 \in \mathbb{C}$ and set $\zeta \coloneqq z_1 - z_2$.
    In the case $z_1 = z_2 = 0$, Remark \ref{rem:last_term} makes the inequality trivial to prove.
    Otherwise, Taylor's formula with integral form gives:
    \begin{equation} \label{eq:taylor_exp_F1}
        F_1 (z_1) = F_1 (z_2) + \langle \zeta, \nabla F_1 (z_2) \rangle + \int_0^1 \langle \zeta, D^2 F_1 (z_2 + t \zeta) \, \zeta \rangle \, (1 - t) \diff t.
    \end{equation}
    The second term of the right-hand side is exactly what we expect since:
    \begin{equation*}
        \langle z, \nabla F_1 (z_2) \rangle = \langle z, 2 z_2 \ln \abs{z_2}^2 \rangle = 2 \, \langle z, z_2 \rangle \, \ln \abs{z_2}^2 = 2 \Re \left( z_2 \overline{z} \right) \, \ln \abs{z_2}^2,
    \end{equation*}
    and is $0$ if $z_2=0$.
    As for the last term of the right-hand side, it can be estimated as previously:
    \begin{equation*}
        \int_0^1 \langle z, D^2 F_1 (z_2 + t z) \, z \rangle \, (1 - t) \diff t \leq 4 \int_0^1 (\ln{\abs{z_2 + t z}} + 1) \, \abs{z}^2 \, (1 - t) \diff t.
    \end{equation*}
    Moreover, there holds for all $t \in [0,1]$
    \begin{equation*}
        \abs{z_2 + t z} = \abs{(1-t) z_2 + t z_1} \leq (1-t) \abs{z_2} + t \abs{z_1} \leq \max( \abs{z_2}, \abs{z_1} )
    \end{equation*}
    Therefore, since $\ln$ is increasing,
    \begin{align*}
        \int_0^1 \langle z, D^2 F_1 (z_2 + t z) \, z \rangle \, (1 - t) \diff t &\leq 4 \, \abs{z}^2 \int_0^1 \biggl( \ln{\Bigl( \max( \abs{z_2}, \abs{z_1} ) \Bigr)} + 1 \biggr) \, (1 - t) \diff t \\
            &\leq 2 \, \abs{z}^2 \biggl( \ln{\Bigl( \max( \abs{z_2}, \abs{z_1} ) \Bigr)} + 1 \biggr)
    \end{align*}
    The conclusion readily follows from putting this inequality into \eqref{eq:taylor_exp_F1}.
\end{proof}

\begin{cor} \label{cor:2nd_exp_en_pot}
    For all $x \in \mathbb{R}$, $t\geq0$, $n \in \mathbb{N}$ and $j \in \{ 1, \dots, N \}$, there holds
    \begin{multline*}
        F_1 (u_n (t',x)) \leq F_1 (G_j (t',x)) + 2 \Re \left( G_j (t',x) \, \overline{w_n^j (t',x)} \right) \, \ln \abs{G_j (t',x)}^2 \\ + 2 \, \abs{w_n^j (t',x)}^2 \, \biggl( \ln{\Bigl( 1 + \abs{w_n^j (t',x)} \Bigr)} + C_0 \biggr),
    \end{multline*}
    where $w_n^j \coloneqq u_n - G_j$.
\end{cor}

\begin{proof}
    Applying Lemma \ref{lem:exp_u_ln_u} with $z_1 = u_n (t',x)$ and $z_2 = G_j (t',x)$ gives
    \begin{multline}
        F_1 (u_n (t',x)) \leq F_1 (G_j (t',x)) + 2 \Re \left( G_j (t',x) \, \overline{w_n^j (t',x)} \right) \, \ln \abs{G_j (t',x)}^2 \\ + \abs{w_n^j (t',x)}^2 \, \biggl( \ln{\Bigl( \max( \abs{G_j (t',x)}, \abs{u_n (t',x)} ) \Bigr)} + 1 \biggr), \label{eq:1st_exp_en_pot}
    \end{multline}
    We know that
    \begin{equation*}
        \abs{u_n (t',x)} \leq \abs{G_j (t',x)} + \abs{w_n^j (t',x)}
    \end{equation*}
    and
    \begin{equation*}
        \norm{G_j (t')}_{L^\infty} \leq C_0.
    \end{equation*}
    Thus,
    \begin{equation*}
        \max( \abs{G_j (t',x)}, \abs{u_n (t',x)} ) \leq C_0 + \abs{w_n^j (t',x)},
    \end{equation*}
    which yields
    \begin{equation*}
        \ln{\Bigl( \max( \abs{G_j (t',x)}, \abs{u_n (t',x)} ) \Bigr)} \leq \ln{( C_0 + \abs{w_n^j (t',x)} )} \leq C_0 + \ln{( 1 + \abs{w_n^j (t',x)} )}.
    \end{equation*}
    The result follows by putting this estimate into \eqref{eq:1st_exp_en_pot}.
\end{proof}

The proof of Lemma \ref{lem:linearize_action} readily follows from expanding $S_j^\textnormal{loc} (t', u_n (t'))$ in terms of $w_n^j (t')$, using Corollary \ref{cor:2nd_exp_en_pot} for the "expansion" of the localized potential energy.

\subsection{Slow variations of the functional}

In this subsection, we prove Proposition \ref{prop:slow_var_loc_func}. Again, this proposition is similar to that in \cite{Martel_Merle__Multi_solitary_waves_NLS,Cote_LeCoz_Multisolitons_NLS} for example, and the proof is almost the same. However, some minor changes occur. The first one comes from the fact that we took a "true" $d$-dimensional partition of unity, and thus the link between the time and space derivatives of this partition is less obvious, yet relatively similar:
\begin{lem}
    For all $j \geq 1$, $t \geq 0$ and $x \in \mathbb{R}$, there holds
    \begin{equation*}
        \abs{\partial_t \psi_j (t', x)} \leq C_0 \, \abs{\nabla \psi_j (t',x)}.
    \end{equation*}
\end{lem}


Thanks to this link, it is easy to prove that $S^\textnormal{loc}$ slowly varies in the same way as in \cite{Martel_Merle__Multi_solitary_waves_NLS,Cote_LeCoz_Multisolitons_NLS}:

\begin{lem}
    For all $t \in [t^\dag, T_n - T'']$, there holds
    \begin{equation*}
        \abs{\frac{\diff}{\diff t} S^\textnormal{loc} (t',u_n(t'))} \leq C_0 \, e^{- \frac{\lambda (v_* t)^2}{4}}.
    \end{equation*}
\end{lem}

\begin{proof}
    We already know that the energy $E (u_n (t))$ is conserved. To estimate the variations of $S(t',u_n(t'))$, we only have to study the variations of the localized masses $M_j (t', u_n(t'))$ and momenta $\mathcal{J}_j (t', u_n(t'))$. Thanks to the expression of the partition of unity, we only need to compute (for $j \geq 1$ only since $\omega_0 = 0$ and $v_0 = 0$)
    \begin{align*}
        \frac{\diff}{\diff t} \int \abs{u_n}^2 \, \psi_j \diff x &= \int \Im \Bigl( \Delta u_n \overline{u_n} \Bigr) \, \psi_j \diff x + \int \abs{u_n }^2 \, \partial_t \psi_j \diff x \\
            &= \int \Im \Bigl( \nabla u_n  \overline{u_n } \Bigr) \cdot \nabla \psi_j \diff x + \int \abs{u_n }^2 \, \partial_t \psi_j \diff x \\
        \abs{\frac{\diff}{\diff t} \int \abs{u_n }^2 \, \psi_j \diff x} &\leq C_0 \int ( \abs{\nabla u_n }^2 + \abs{u_n }^2 ) \, \abs{\nabla \psi_j} \diff x.
    \end{align*}
    Similarly, we have for $\mathcal{J}_j (t',u_n(t'))$
    \begin{multline*}
        \frac{\diff}{\diff t} \int \Im \Bigl( \nabla u_n (t') \, \overline{u_n (t')} \Bigr) \psi_j (t') \diff x = \int \Re \Bigl( \Bigl( \overline{\nabla u_n} \cdot \nabla \psi_j (t') \Bigr) \, \nabla u_n \Bigr) \diff x
            - \int \Im \Bigl( \nabla u_n (t') \, \overline{u_n (t')} \Bigr) \, \partial_t \psi_j (t') \diff x \\
            - \lambda \int \abs{u_n (t')}^2 \, \nabla \psi_j (t') \diff x
            - \frac{1}{4} \int \abs{u_n (t')}^2 \, \nabla \cdot D_{xx}^2 \psi_j (t') \diff x.
    \end{multline*}
    Therefore there holds again
    \begin{equation*}
        \abs{\frac{\diff}{\diff t} \mathcal{J}_j (t', u_n (t'))} \leq C_0 \int \Bigl( \abs{\nabla u_n (t')}^2 + \abs{u_n (t')}^2 \Bigr) \, \abs{\nabla \psi_j(t')} \diff x + C_0 \int \abs{u_n (t')}^2 \, \norm{D_{xxx}^3 \psi_j(t')} \diff x.
    \end{equation*}
    Now, remark that
    \begin{equation*}
        \int \Bigl( \abs{\nabla u_n (t')}^2 + \abs{u_n (t')}^2 \Bigr) \, \abs{\nabla \psi_j(t')} \diff x
        \leq 2 \biggl( \int \Bigl( \abs{\nabla G (t')}^2 + \abs{G (t',x)}^2 \Bigr) \, \abs{\nabla \psi_j(t')} \diff x + C_0 \, \norm{w_n (t')}_{H^1 (\mathbb{R})}^2 \biggr)
    \end{equation*}
    and
    \begin{equation*}
        \int \abs{u_n (t',x)}^2 \, \norm{D_{xxx}^3 \psi_j(t')} \diff x \\
        \leq 2 \biggl( \int \abs{G (t',x)}^2 \, \norm{D_{xxx}^3 \psi_j(t')} \diff x + C_0 \, \norm{w_n (t')}_{L^2 (\mathbb{R})}^2 \biggr).
    \end{equation*}
    By assumption, we know that for all $t \in [t^\dag, T_n - T'']$
    \begin{equation*}
        \norm{w_n (t')}_{H^1 (\mathbb{R})}^2 \leq 2 e^{- \frac{\lambda (v_* t)^2}{2}}.
    \end{equation*}
    Moreover, by Lemma \ref{lem:est_H1_gauss}, we have
    \begin{gather*}
        \int \Bigl( \abs{\nabla G (t',x)}^2 + \abs{G (t',x)}^2 \Bigr) \, \abs{\nabla \psi_j(t')} \diff x = o \biggl( e^{- \frac{\lambda (v_* t)^2}{2}} \biggr), \\
        \int \abs{G (t',x)}^2 \, \norm{D_{xxx}^3 \psi_j(t')} \diff x = o \biggl( e^{- \frac{\lambda (v_* t)^2}{2}} \biggr).
    \end{gather*}
    Consequently,
    \begin{equation*}
        \abs{\frac{\diff}{\diff t} \int \abs{u_n (t')}^2 \, \psi_j (t') \diff x} + \abs{\frac{\diff}{\diff t} \int \Im \Bigl( \nabla u_n (t') \, \overline{u_n (t')} \Bigr) \psi_j (t') \diff x} \leq C_0 \, e^{- \frac{\lambda (v_* t)^2}{2}}
    \end{equation*}
    Plugging this estimate into the expression of $M_j$ and $\mathcal{J}_j$ gives for all $t \geq t^\dag$
    \begin{equation*}
        \abs{\frac{\diff}{\diff t} M_j (t', u(t'))} + \abs{\frac{\diff}{\diff t} \mathcal{J}_j (t', u(t'))} \leq C_0 \, e^{- \frac{\lambda (v_* t)^2}{2}},
    \end{equation*}
    and the conclusion readily follows.
\end{proof}

The fact that the convergence is Gaussian instead of exponential gives a free $t^{-1}$ factor when integrating, which is enough to be negligible.

\begin{cor} \label{cor:variation_energy}
    There holds for $n$ large enough and $t \in [t^\dag, T_n - T'']$ :
    \begin{equation*}
        \abs{S^\textnormal{loc} (t',u_n(t')) - S^\textnormal{loc} (T_n,G(T_n))} \leq C_0 \, t^{-1} e^{- \frac{\lambda (v_* t)^2}{2}}.
    \end{equation*}
\end{cor}

\begin{proof}
    Defining $\tilde{S}_n (t) = S(t,u_n(t))$, we can estimate this difference thanks to the previous estimate:
    \begin{align*}
        S(t',u_n(t')) - S(T_n,G(T_n)) &= \int_{t'}^{T_n} \frac{\diff \tilde{S}_n}{\diff t} (s) \diff s = \int_{t}^{T_n - T''} \frac{\diff \tilde{S}_n}{\diff t} (T'' + s) \diff s \\
        \abs{S(t',u_n(t')) - S(T_n,G(T_n))} &\leq \int_{t}^{T_n - T''} C_0 \, e^{- \frac{\lambda (v_* s)^2}{2}} \diff s \\
            &\leq C_0 \, t^{-1} e^{- \frac{\lambda (v_* t)^2}{2}},
    \end{align*}
    thanks to Lemma \ref{lem:int_error_gauss}.
\end{proof}

In the previous result, we have $S^\textnormal{loc} (T_n,G(T_n))$: we would like to have $S_j (G_j)$ instead. Thanks to Corollary \ref{cor:diff_actions_for_gauss}, we only need $S^\textnormal{loc} (t', G_j (t'))$.

\begin{lem} \label{lem:energy_sum_soliton}
    There holds for all $t \geq 0$ and $j \in \{ 1, \dots, N \}$,
    \begin{gather*}
        \abs{S_j^\textnormal{loc} (t',G(t')) - S_j^\textnormal{loc} (t', G_j (t'))} \leq C_0 \, t^{-1} e^{- \frac{\lambda (v_* t)^2}{2}}, \\
        \abs{S_0^\textnormal{loc} (t',G(t'))} \leq C_0 \, t^{-1} e^{- \frac{\lambda (v_* t)^2}{2}}.
    \end{gather*}
\end{lem}

\begin{proof}
    Decomposing $S_j^\textnormal{loc} (t',G(t'))$, we get
    \begin{multline*}
        S_j^\textnormal{loc} (t',G(t')) - S_j^\textnormal{loc} (G_j (t')) = \frac{1}{2} \Bigl( \int \abs{\nabla G (t')}^2 \psi_j (t') \diff x - \int \abs{\nabla G_j (t')}^2 \psi_j (t') \diff x \Bigr) \\
            \begin{aligned}
            &- \lambda \biggl( \int \abs{G (t')}^2 \ln \abs{G (t')}^2 \psi_j (t') \diff x - \int \abs{G_j (t')}^2 \ln \abs{G_j (t')}^2 \psi_j (t') \diff x \biggr) \\
            &+ \Bigl( 2 \lambda \omega_j + \lambda + \frac{\abs{v_j}^2}{2} \Bigr) (M_j (t', G(t')) - M_j (t',G_j (t'))) \\ &- v_j \cdot ( \mathcal{J}_j (t', G(t')) - \mathcal{J}_j (t',G_j (t')) ).
            \end{aligned}
    \end{multline*}

    \begin{itemize}[leftmargin=*]
    \item For the first term, decomposing $G = \sum_j G_j$,
    \begin{multline*}
        \int \abs{\nabla G (t')}^2 \psi_j (t') \diff x - \int \abs{\nabla G_j (t')}^2 \psi_j (t') \diff x \\ \begin{aligned} &= \sum_{(k,\ell) \neq (j,j)} \Re \int \nabla G_\ell (t') \cdot \overline{\nabla G_k (t')} \, \psi_j (t') \diff x \\
            &= \sum_{k \neq \ell} \Re \int \nabla G_\ell (t') \cdot \overline{\nabla G_k (t')} \, \psi_j (t') \diff x + \sum_{k \neq j} \int \abs{\nabla G_k (t')}^2 \, \psi_j (t') \diff x, \end{aligned}
    \end{multline*}
    and the conclusion with Lemma \ref{lem:est_H1_gauss} and Lemma \ref{lem:orth_gauss}.

    \item For the third and fourth terms, the same kind of decomposition can be used:
    \begin{gather*}
        \mathcal{J}_j (t', G(t')) - \mathcal{J}_j (t',G_j) = \sum_{(k,\ell) \neq (j,j)} \Im \int \nabla G_k (t',x) \, \overline{G_l (t',x)} \, \psi_j(t',x) \diff x, \\
        M_j (t', G(t')) - M_j (t',G_j) = \sum_{(k,\ell) \neq (j,j)} \Re \int G_k (t',x) \, \overline{G_l (t',x)} \, \psi_j(t',x) \diff x.
    \end{gather*}
    The conclusion comes in the same way.
    
    \item For the second term, we use a similar computation as in \cite[Lemma~3.2]{Ferriere__superposition_logNLS}. Precisely, we will use the following lemma:
    
    \begin{lem}[{\cite[Lemma~3.3]{Ferriere__superposition_logNLS}}] \label{lem:lip_z_log_z}
        Set $F(z) \coloneqq z \ln{\abs{z}}$.
        For all $z, \tilde{z} \in \mathbb{C}$ such that $\abs{z} \leq 1$, $\abs{\tilde{z}} \leq 1$ and $z \neq 0$, there holds
        \begin{equation*}
            \abs{F(\tilde{z}) - F(z)} \leq \abs{z - \tilde{z}} \Bigl[ 3 - \ln \abs{z} \Bigr].
        \end{equation*}
    \end{lem}

    Then, by changing $G_j$ and $G$ into $\tilde{G}_j \coloneqq N^{-1} e^{- \omega} \, G_j$ and $\tilde{G} \coloneqq N^{-1} e^{- \omega} \, G$ respectively (with $\omega \coloneqq \max_k \omega_k$) and for all $j$, we get
    \begin{equation*}
        \tilde{G} = \sum_k \tilde{G}_k, \qquad
        \sum_k \abs{\tilde{G}_k} \leq 1.
    \end{equation*}
    
    Thus, for any $j \in \{ 1, \dots, N \}$, all $t \geq T$ and all $x \in \mathbb{R}$, there holds
    \begin{multline*}
        \abs{\abs{G (t',x)}^2 \ln \abs{G (t',x)}^2 - \abs{G_j (t',x)}^2 \ln \abs{G_j (t',x)}^2} \\
        \begin{aligned}
            &= N^2 e^{2 \omega} \abs{\abs{\tilde{G} (t',x)}^2 \Bigl( 2 \omega + \ln{N^2} + \ln \abs{\tilde{G} (t',x)}^2 \Bigr) - \abs{\tilde{G}_j (t',x)}^2 \Bigl( 2 \omega + \ln{N^2} + \ln \abs{\tilde{G}_j (t',x)}^2 \Bigr)} \\
            &\leq N^2 e^{2 \omega} \abs{\abs{\tilde{G} (t',x)}^2 - \abs{\tilde{G}_j (t',x)}^2} \Bigl[ 2 \abs{\omega} + \ln{N^2} + 3 - \ln \abs{\tilde{G}_j (t',x)}^2 \Bigr] \\
            &\leq \abs{\abs{G (t',x)}^2 - \abs{G_j (t',x)}^2} \Bigl[ 2 \abs{\omega} + \ln{N^2} + 3 - \ln \abs{\tilde{G}_j (t',x)}^2 \Bigr] \\
            &\leq C_0 \, \abs{G (t',x) - G_j (t',x)} \Bigl( \abs{G_j (t',x)} + \sum_k \abs{G_k (t',x)} \Bigr) \Bigl[ 1 + \abs{x - x_j^* (t')}^2 \Bigr] \\
            &\leq C_0 \sum_{\ell \neq j} \sum_k \abs{G_\ell (t', x)} \abs{G_k (t',x)} (1 + \abs{x - x_j^* (t')}^2). 
        \end{aligned}
    \end{multline*}
    Multiplying by $\psi_j (t')$ and integrating over $\mathbb{R}^d$, we get
    \begin{multline*}
        \abs{\int \abs{G (t')}^2 \ln \abs{G (t')}^2 \, \psi_j (t') \diff x - \sum_j \int \abs{G_j (t')}^2 \ln \abs{G_j (t')}^2 \, \psi_j (t') \diff x} \\
        \leq C_0 \sum_{\ell \neq j} \sum_k \int \abs{G_\ell (t')} \abs{G_k (t')} (1 + \abs{x - x_j^* (t')}^2) \, \psi_j (t',x) \diff x.
    \end{multline*}
    Thus, Lemma \ref{lem:est_H1_gauss} and Lemma \ref{lem:orth_gauss} yield the conclusion for this term.
    \end{itemize}

    The conclusion readily follows for the first inequality.
    As for the second one, the same computations can be done for $S_0^\textnormal{loc} (t',G(t')) - S_0^\textnormal{loc} (t', G_k (t'))$ for some $k$, so that
    \begin{equation*}
        \abs{S_0^\textnormal{loc} (t',G(t')) - S_0^\textnormal{loc} (t', G_k (t'))} \leq C_0 \, t^{-1} e^{- \frac{\lambda (v_* t)^2}{2}}
    \end{equation*}
    Moreover, it can easily be proved (thanks to Lemma \ref{lem:est_H1_gauss}) that
    \begin{equation*}
        \abs{S_0^\textnormal{loc} (t', G_k (t'))} \leq C_0 \, t^{-1} e^{- \frac{\lambda (v_* t)^2}{2}},
    \end{equation*}
    and therefore the conclusion.
\end{proof}

We now have all the results we need to prove Proposition \ref{prop:slow_var_loc_func}.

\begin{proof}[Proof of Proposition \ref{prop:slow_var_loc_func}]
    We decompose the left-hand side in order to be able to apply the previous results:
    \begin{multline*}
        S^\textnormal{loc} (t',u_n(t')) - \sum_j S(G_j) = \Bigl( S^\textnormal{loc} (t',u_n(t')) - S^\textnormal{loc} (T_n,G(T_n)) \Bigr) \\ \begin{aligned}
            &+ S^\textnormal{loc}_0 (T_n,G(T_n))
            + \sum_{j \geq 1} \Bigl( S^\textnormal{loc}_j (T_n,G(T_n)) - S^\textnormal{loc}_j (T_n,G_j(T_n)) \Bigr) \\ &+ \sum_{j \geq 1} \Bigl( S^\textnormal{loc}_j (T_n,G_j(T_n)) - S_j (G_j) \Bigr).
            \end{aligned}
    \end{multline*}
    Thanks to Corollaries \ref{cor:variation_energy} (for the first line of the right-hand side) and \ref{cor:diff_actions_for_gauss} (for the last one) and to Lemma \ref{lem:energy_sum_soliton} (for the second one) along with the fact that
    \begin{equation*}
        (T_n - T'')^{-1} e^{- \frac{\lambda (v_* (T_n - T''))^2}{2}} \leq t^{-1} e^{- \frac{\lambda (v_* t)^2}{2}}, \qquad \text{for all } 0 < t \leq T_n - T''
    \end{equation*}
    as soon as $T_n \geq T''$, we get
    \begin{equation*}
        S^\textnormal{loc} (t',u_n(t')) - \sum_j S(G_j) \leq C_0 \, t^{-1} e^{- \frac{\lambda (v_* t)^2}{2}}. \qedhere
    \end{equation*}
\end{proof}

\section{Uniform $\mathcal{F} (H^1)$-estimates} \label{sec:FH1_est}

The final step of our proof is the uniform estimates in $\mathcal{F} (H^1)$. The proof relies on an improvement of the computation of Section \ref{sec:L2_est}, using the uniform estimates in $H^1$ which are now proved.

\begin{prop}
    For all $n \in \mathbb{N}$ such that $T_n \geq T$ and for all $t \in [0, T_n - T]$ (with $\check{t} \coloneqq T + t$), there holds
    \begin{equation*}
        \norm{w_n (\check{t})}_{\mathcal{F} (\dot{H}^1)} \leq C_0 \, e^{- \frac{\lambda (v_* t)^2}{4}}.
    \end{equation*}
\end{prop}

\begin{proof}
    We recall that $w_n = u_n - G$ satisfies
    \begin{equation*}
        i \partial_t w_n + \frac{1}{2} \Delta w_n = - \lambda \Bigl[ u_n \ln \abs{u_n}^2 - \sum_k G_k \ln \abs{G_k}^2 \Bigr], \qquad w_n (T_n) = 0,
    \end{equation*}
    and there now holds for all $t \in [0, T_n - T]$
    \begin{equation*}
        \norm{w_n (\check{t})}_{L^2} \leq e^{- \frac{\lambda (v_* t)^2}{4}}, \qquad
        \norm{w_n (\check{t})}_{\dot{H}^1} \leq e^{- \frac{\lambda (v_* t)^2}{4}}.
    \end{equation*}
    We also recall that we took $T$ large enough so that for all $j \in \{ 1, \dots, N-1 \}$ and $t \geq 1$, we have
    \begin{equation*}
        \abs{x_{j+1} - x_j + (v_{j+1} - v_j) \check{t}} \geq \varepsilon_0^{-1} + v_* (t + \tau),
    \end{equation*}
    as a consequence of \eqref{eq:distance_x_k}.
    We compute the variations of this quantity:
    \begin{multline}
        \frac{\diff}{\diff t} \int \abs{x}^2 \abs{w_n (\check{t})}^2 \diff x = - \int \abs{x}^2 \Im \Bigl[ \Delta w_n (\check{t}) \overline{w_n (\check{t})} \Bigr] \diff x \\
        - 2 \lambda \int \abs{x}^2 \Im \biggl[ \Bigl[ u_n (\check{t}) \ln \abs{u_n (\check{t})}^2 - \sum_k G_k (\check{t}) \ln \abs{G_k (\check{t})}^2 \biggr] \overline{w_n (\check{t})} \Bigr] \diff x. \label{comput_var_poids_quadr}
    \end{multline}
    \begin{itemize}
    \item For the first term, performing an integration by parts, there holds
    \begin{equation*}
        - \int \abs{x}^2 \Im \Bigl[ \Delta w_n (\check{t}) \overline{w_n (\check{t})} \Bigr] \diff x = 2 \int x \cdot \Im \Bigl[ \nabla w_n (\check{t}) \overline{w_n (\check{t})} \Bigr] \diff x.
    \end{equation*}
    This is easy to estimate with a Cauchy-Schwarz inequality:
    \begin{align*}
        \abs{\int x \cdot \Im \Bigl[ \nabla w_n (\check{t}) \overline{w_n (\check{t})} \Bigr] \diff x} &\leq \biggl( \int \abs{\nabla w_n (\check{t})}^2 \diff y \biggr)^\frac{1}{2} \biggl( \int \abs{x}^2 \abs{w_n (\check{t})}^2 \diff x \biggr)^\frac{1}{2} \\
            &\leq e^{- \frac{\lambda (v_* t)^2}{4}} \biggl( \int \abs{x}^2 \abs{w_n (\check{t})}^2 \diff x \biggr)^\frac{1}{2}.
    \end{align*}
    Thus, we have
    \begin{equation}
        \abs{\int \abs{x}^2 \Im \Bigl[ \Delta w_n (\check{t}) \overline{w_n (\check{t})} \Bigr] \diff x} \leq C_0 \, e^{- \frac{\lambda (v_* t)^2}{4}} \biggl( \int \abs{x}^2 \abs{w_n (\check{t})}^2 \diff x \biggr)^\frac{1}{2}. \label{1st_term_var_poids_quadr}
    \end{equation}

    \item For the last term, we use again Lemma \ref{lem_log_inequality} and the fact that $w_n = u_n - G$, so that
    \begin{multline*}
        \abs{\Im \biggl[ \Bigl[ u_n \ln \abs{u_n}^2 - \sum_k G_k \ln \abs{G_k}^2 \biggr] \overline{w_n} \Bigr]} \\
            \begin{aligned}
                &\leq \abs{\Im \biggl[ \Bigl[ u_n \ln \abs{u_n}^2 - G \ln \abs{G}^2 \biggr] \overline{w_n} \Bigr]} + \abs{\Im \biggl[ \Bigl[ G \ln \abs{G}^2 - \sum_k G_k \ln \abs{G_k}^2 \biggr] \overline{w_n} \Bigr]} \\
                &\leq 2 \abs{w_n}^2 + \abs{ G \ln \abs{G}^2 - \sum_k G_k \ln \abs{G_k}^2 } \abs{w_n}.
            \end{aligned}
    \end{multline*}
    Thus, we get
    \begin{multline*}
        \abs{\int \abs{x}^2 \Im \biggl[ \Bigl[ u_n \ln \abs{u_n}^2 - \sum_k G_k \ln \abs{G_k}^2 \biggr] \overline{w_n} \Bigr] \diff x} \leq 2 \int \abs{x}^2 \abs{w_n (\check{t})}^2 \diff x \\ + \int \abs{x}^2 \abs{ G \ln \abs{G}^2 - \sum_k G_k \ln \abs{G_k}^2 } \abs{w_n (\check{t})} \diff x
    \end{multline*}
    For the second term of the right-hand side, performing a Cauchy-Schwarz inequality leads to
    \begin{multline*}
        \int \abs{x}^2 \abs{ G \ln \abs{G}^2 - \sum_k G_k \ln \abs{G_k}^2 } \abs{w_n (\check{t})} \diff x \\
            \begin{aligned}
                &\leq \biggl( \int \abs{x}^2 \abs{ G \ln \abs{G}^2 - \sum_k G_k \ln \abs{G_k}^2 }^2 \diff x \biggr)^\frac{1}{2} \biggl( \int \abs{x}^2 \abs{w_n (\check{t})}^2 \diff x \biggr)^\frac{1}{2}
            \end{aligned}
    \end{multline*}
    For the first factor, we use the following result whose proof is postponed to Appendix \ref{app:x_u_log_u}.
    
    \begin{prop} \label{prop:x_u_log_u}
        For all $t \geq 0$, there holds
        \begin{equation*}
            \norm{\abs{x} \abs{ G \ln \abs{G}^2 - \sum_k G_k \ln \abs{G_k}^2 }}_{L^2} \leq C_0 \, e^{- \frac{\lambda (v_* t)^2}{4}}
        \end{equation*}
    \end{prop}

    Thus we have
    \begin{multline}
        \abs{\int \abs{x}^2 \Im \biggl[ \Bigl[ u_n \ln \abs{u_n}^2 - \sum_k G_k \ln \abs{G_k}^2 \biggr] \overline{w_n} \Bigr] \diff x} \leq 2 \int \abs{x}^2 \abs{w_n (\check{t})}^2 \diff x \\ + C_0 \, e^{- \frac{\lambda (v_* t)^2}{4}} \biggl( \int \abs{x}^2 \abs{w_n (\check{t})}^2 \diff x \biggr)^\frac{1}{2}. \label{2nd_term_var_poids_quadr}
    \end{multline}
    \end{itemize}

    Plugging \eqref{1st_term_var_poids_quadr} and \eqref{2nd_term_var_poids_quadr} into \eqref{comput_var_poids_quadr}, and dividing the whole inequality by $\norm{\abs{x} \abs{w_n (\check{t})} }_{L^2}$, we obtain
    \begin{equation*}
        \abs{\frac{\diff}{\diff t} \norm{\abs{x} \abs{w_n (\check{t})} }_{L^2}} \leq
        C_0 \, e^{- \frac{\lambda (v_* t)^2}{4}}
        + 2 \lambda \norm{\abs{x} \abs{w_n (\check{t})} }_{L^2}.
    \end{equation*}
    Hence, in the same way as in the proof of Proposition \ref{prop:unif_est_L2}, the Gronwall lemma backward in time between $T_n$ and $\check{t}$ and the fact that $w_n (T_n) = 0$ yields for all $t \in [0, T_n - T]$ (and still with $\check{t} = T + t$):
    \begin{equation*}
        \norm{w_n (\check{t}) }_{\mathcal{F} (H^1)} \leq C_0 \, e^{- \frac{\lambda (v_* t)^2}{4}}. \qedhere
    \end{equation*}
\end{proof}

\section{Compactness for the multi-gaussian} \label{sec:compactness_multi_br}

This section is devoted to the proof of the Compactness property \ref{prop:compactness_br} for the multi-gaussian case. Since we do not have any bound for the $\mathcal{F}(H^1)$ norm, the proof is here completely similar to that in \cite{Martel_Merle__Multi_solitary_waves_NLS, Cote_LeCoz_Multisolitons_NLS} for example. However, not only in order to be able to perform the same proof but also in order to have a limit in $W$, we need $u_n (T)$ to be uniformly bounded in $W$.

\begin{lem} \label{lem:W_bound}
    $E(u_n)$ is uniformly bounded in $n$. In particular, $u_n$ is uniformly bounded in $\mathcal{C}_b (\mathbb{R}, W)$.
\end{lem}

\begin{proof}
Since we know that the energy is independent in time, we only have to prove for the first part that $E(B(T_n))$ is bounded.
First of all, the $H^1$ norm of $B(T_n)$ is obviously bounded since:
\begin{equation*}
    \norm{B (T_n)}_{H^1} \leq \sum_k \norm{B_k (T_n)}_{H^1} \leq \sum_k \left( \norm{B_k (T_n)}_{L^2} + \norm{\nabla B_k (T_n)}_{L^2} \right).
\end{equation*}
Moreover, $B(T_n)$ is also bounded in $L^1$:
\begin{align*}
    \norm{B (T_n)}_{L^1} \leq \sum_k \norm{B_k (T_n)}_{L^1}
    &\leq C_0 \sum_k \norm{\exp \left[ - \frac{1}{2} (x - x_k - v_k T_n)^\top \Re A(t) (x - x_k - v_k T_n) \right]}_{L^1} \\
    &\leq C_0 \norm{e^{-\frac{\sigma_- \abs{x}^2}{2}}}_{L^1} \leq C_0.
\end{align*}
Therefore, we claim that $\int \abs{B(T_n)}^2 \ln \abs{B(T_n)}^2$ is uniformly bounded and so is $E(B(T_n))$.
Indeed, there holds
\begin{equation*}
    \int \abs{B(T_n)}^2 \abs{\ln \abs{B(T_n)}^2} \leq C_0 \left( \int \abs{B(T_n)} + \int \abs{B(T_n)}^{2+\frac{1}{d}} \right) \leq C_0 \left( \norm{B(T_n)}_{L^1} + \norm{B(T_n)}_{H^1}^{2+\frac{1}{d}} \right),
\end{equation*}
thanks to Sobolev embedding. Therefore, $E(u_n) = E(B(T_n))$ is uniformly bounded.
Thus, we can derive that $\nabla u_n (t)$ is uniformly bounded both in $n$ and $t$ from this and the fact that
\begin{equation*}
    E_n^+ (t) \coloneqq \frac{1}{2} \norm{\nabla u_n (t)}_{L^2}^2 + \lambda \norm{u_n (t)}_{L^2}^2 + \lambda \int_{\abs{u_n (t)} \leq 1} \abs{u_n (t)}^2 \abs{\ln{\abs{u_n(t)}^2}} \diff x
\end{equation*}
satisfies
\begin{equation}
    0 \leq \frac{1}{2} \norm{\nabla u_n (t)}_{L^2}^2 \leq E_n^+ (t) = E(u_n) + \lambda \int_{\abs{u_n (t)} > 1} \abs{u_n (t)}^2 \abs{\ln{\abs{u_n(t)}^2}} \diff x. \label{eq:pos_en_ineq}
\end{equation}
Indeed, by Gagliardo-Nirenberg inequality, we have
\begin{align*}
    L_n^+ (t) = \int_{\abs{u_n (t)} > 1} \abs{u_n (t)}^2 \abs{\ln{\abs{u_n(t)}^2}} \diff x &\leq C_0 \int \abs{u_n (t)}^{2 + \frac{1}{2d}} \diff x \\
    &\leq C_0 \norm{u_n (t)}_{L^2}^{1+\frac{1}{2d}} \, \norm{\nabla u_n (t)}_{L^2} \\
    &\leq C_0 \, \norm{\nabla u_n (t)}_{L^2},
\end{align*}
since the $L^2$ norm of $u_n$ is uniformly bounded.
Thus, putting this into \eqref{eq:pos_en_ineq} leads to:
\begin{equation*}
    \frac{1}{2} \norm{\nabla u_n (t)}_{L^2}^2 \leq C_0 ( 1 + \norm{\nabla u_n (t)}_{L^2} ).
\end{equation*}
Hence, $\norm{\nabla u_n (t)}_{L^2}$ is bounded uniformly in $n$ and $t$, and so is $L_n^+ (t)$, therefore so is $E_n^+ (t)$. This yields that
\begin{equation*}
    \int \abs{u_n (t)}^2 \abs{\ln{\abs{u_n(t)}^2}} \diff x \textnormal{ is bounded uniformly in $t$ and $n$.}
\end{equation*}
Thus, $u_n (t)$ is bounded in $W ( \mathbb{R}^d )$ uniformly in $t$ and $n$.
\end{proof}

\begin{proof}[Proof of Proposition \ref{prop:compactness_br}]
    We already have a uniform boundedness of $u_n (T)$ in $H^1$. In order to get compactness in $L^2$, we shall prove that $u_n (T)$ is compact at infinity. Choose $\delta > 0$. We want to show that there exists $r_\delta$ such that for any $n$ we have
    \begin{equation*}
        \int_{\abs{x} > r_\delta} \abs{u_n (T,x)}^2 \diff x < \delta.
    \end{equation*}
    Let $T_\delta \geq 1$ such that 
    \begin{equation*}
        e^{- \frac{\sigma_- (v_* T_\delta)^2}{4}} \leq \sqrt{\frac{\delta}{8}},
    \end{equation*}
    so that, thanks to Proposition \ref{prop:unif_est_br}, there holds for any $n \in \mathbb{N}$,
    \begin{equation*}
        \norm{u_n (T + T_\delta,.) - B (T + T_\delta,.)}_{L^2}^2 \leq \frac{\delta}{8}.
    \end{equation*}
    The members of $B$ are Gaussians, so we can find $\bar{r}_\delta$ such that
    \begin{equation*}
        \int_{\abs{x} > \bar{r}_\delta} \abs{B (T + T_\delta, x)}^2 \diff x < \frac{\delta}{8}.
    \end{equation*}
    Therefore we can infer from the two previous estimates that for all $n \in \mathbb{N}$
    \begin{equation*}
        \int_{\abs{x} > \bar{r}_\delta} \abs{u_n (T + T_\delta, x)}^2 \diff x < \frac{\delta}{2}.
    \end{equation*}
    To transfer this property up to $T$, we use a virial argument. Take $\hat{r}_\delta > 0$ to be fixed later and a $\mathcal{C}^1$ cut-off function $\chi : \mathbb{R} \rightarrow \mathbb{R}$ such that
    \begin{equation*}
        \chi (s) = 0 \text{ for } s < 0, \qquad \chi (s) = 1 \text{ for } s > 1, \qquad \chi (s) \in [0, 1] \text{ for } s \in \mathbb{R}.
    \end{equation*}
    Now set
    \begin{equation*}
        V(t) \coloneqq \int \abs{u_n (t)}^2 \, \chi \left( \frac{\abs{x} - \bar{r}_\delta}{\hat{r}_\delta} \right) \diff x.
    \end{equation*}
    Thanks to the previous estimate, we already know that $V(T + T_\delta) < \frac{\delta}{2}$.
    Moreover, since $u_n$ satisfies \eqref{foc_log_nls}, it is easy to compute
    \begin{equation*}
        V'(t) = \frac{2}{\hat{r}_\delta} \int \Im \left( \overline{u_n} (t,x) \frac{x}{\abs{x}} \cdot \nabla u_n (t,x) \right) \chi' \left( \frac{\abs{x} - \bar{r}_\delta}{\hat{r}_\delta} \right) \diff x.
    \end{equation*}
    Thanks to Lemma \ref{lem:W_bound}, we know that $u_n(t)$ is uniformly bounded in $H^1$. Hence, the previous integral is uniformly bounded and thus we have
    \begin{equation*}
        \abs{V'(t)} \leq \frac{C_0}{\hat{r}_\delta}.
    \end{equation*}
    We choose now $\hat{r}_\delta$ such that $\frac{C_0}{\hat{r}_\delta} T_\delta < \frac{\delta}{2}$. Hence,
    \begin{equation*}
        \abs{V (T) - V (T + T_\delta)} < \frac{\delta}{2},
    \end{equation*}
    and therefore
    \begin{equation*}
        V (T) < \delta.
    \end{equation*}
    We infer from the definition of $\chi$ and with $r_\delta = \bar{r}_\delta + \hat{r}_\delta$ that for all $n \in \mathbb{N}$,
    \begin{equation*}
        \int_{\abs{x} > r_\delta} \abs{u_n (T_\textnormal{in}, x)}^2 \diff x < \delta,
    \end{equation*}
    which is the desired conclusion.

    Therefore, we get compactness in $L^2$: there exists a $u_\textnormal{in} \in H^1$ such that $u_n (T) \rightarrow u_\textnormal{in}$ (up to a subsequence) in $L^2$ as $n \rightarrow \infty$. Moreover, $u_n (T)$ is uniformly bounded in $W$ which is a reflexive Banach space when endowed with a Luxembourg type norm (see \cite{Cazenave_log_nls}), so $u_\textnormal{in} \in W$.
\end{proof}

\section{Rigidity property} \label{sec:rigidity}

In this section, we prove the claims made in Remarks \ref{rem:rigidity} and \ref{rem:rigidity_br}, which can be summarized as follows:

\begin{lem} \label{lem:rigidity_multi}
    For any solution $v$ to \eqref{foc_log_nls}, either $v = u$ the multi-gaussian constructed above, either there exists $T_1 > 0$ and $C_1 > 0$ such that there holds for all $t \geq T_1$
    \begin{equation*}
        \norm{v(t) - \sum B_k (t)}_{L^2} \geq C_1 e^{- 2 \lambda t}.
    \end{equation*}
\end{lem}

This lemma is actually a corollary of a relatively more general result: a rigidity property for any solution to \eqref{foc_log_nls}.

\begin{lem} \label{lem:rigidity_all}
    For any solutions $v_1$ and $v_2$ to \eqref{foc_log_nls}, either $v_1 = v_2$ or there exists $C_2 > 0$ such that for all $t \geq 0$
    \begin{equation*}
        \norm{v_1 (t) - v_2 (t)}_{L^2} \geq C_2 \, e^{-2 \lambda t}.
    \end{equation*}
\end{lem}

This lemma is itself an obvious corollary of Lemma \ref{lem:L2_energy_est}, and we show how it leads to Lemma \ref{lem:rigidity_multi}.

\begin{proof}
    Take $v$ solution to \eqref{foc_log_nls} and suppose $v \neq u$ where $u$ is the multi-gaussian constructed above. Then, Lemma \ref{lem:rigidity_all} gives some $C_2 > 0$ such that
    \begin{equation*}
        \norm{v (t) - u (t)}_{L^2} \geq C_2 \, e^{-2 \lambda t}.
    \end{equation*}
    Thanks to \eqref{eq:dec_est_th_br}, we get for all $t \geq 0$:
    \begin{align*}
        \norm{v(T + t) - \sum B_k (T + t)}_{L^2} &\geq \norm{v (T + t) - u (T + t)}_{L^2} - \norm{u(T + t) - \sum B_k (T + t)}_{L^2} \\
            &\geq C_2\, e^{-2 \lambda T} \, e^{-2 \lambda t} - e^{- \frac{\sigma_- (v_* t)^2}{4}},
    \end{align*}
    and the conclusion readily follows.
\end{proof}

\newpage

\appendix

\section{Proof of Lemma \ref{lem:est_H1_gauss}} \label{sec_app:proof_lem_gauss}

Before proving this lemma, we prove the following results which will be useful for the last estimate:

\begin{lem} \label{lem:gauss_error_2}
    For all $n \in \mathbb{N}$, there exists $C_n > 0$ such that for all $\gamma > 0$ and $R \geq \gamma^{-\frac{1}{2}}$, there holds
    \begin{gather*}
        I_n \coloneqq \int_R^{\infty} x^{n} e^{-\gamma x^2} \diff x \leq C_n \frac{R^{n-1}}{\gamma} e^{-\gamma R^2}.
    \end{gather*}
\end{lem}

\begin{proof}[Proof of Lemma \ref{lem:gauss_error_2}]
    The case $n=0$ is exactly Lemma \ref{lem:int_error_gauss}. The case $n=1$ easily follows from
    \begin{equation*}
        I_1 = \frac{1}{2 \gamma} e^{-\gamma R^2}.
    \end{equation*}
    For the case $n \geq 2$, we get
    \begin{align*}
        I_n &\leq \int_R^{\infty} x^{n-1} \cdot x \, e^{-\gamma \abs{x}^2} \diff x \\
            &\leq \Bigl[ - \frac{x^{n-1}}{2 \gamma} e^{-\gamma \abs{x}^2} \Bigr]_R^\infty + (n - 1) \int_R^{\infty} \frac{x^{n-2}}{2 \gamma} e^{-\gamma \abs{x}^2} \diff x \\
            &\leq \frac{R^{n-1}}{2 \gamma} e^{-\gamma R^2} + \frac{2n - 1}{2 \gamma} I_{n-2} \\
            &\leq \frac{R^{n - 1}}{2 \gamma} e^{-\gamma R^2} + \frac{2n - 1}{2} R^2 \, I_{n-2},
    \end{align*}
    since $\frac{1}{\gamma} \leq R^2$.
    The conclusion readily follows from a simple induction.
\end{proof}

\begin{lem} \label{lem:gauss_error_3}
    For all $n \in \mathbb{N}$ and $d \in \mathbb{N}^*$, there exists $C_{d,n} > 0$ such that for all $\gamma > 0$ and $R \geq \gamma^{-\frac{1}{2}}$, there holds
    \begin{gather*}
        M_n \coloneqq \int_{\mathcal{B}_R^\complement} \abs{x}^{n} e^{-\gamma \abs{x}^2} \diff x \leq C_n \frac{R^{d+n-2}}{\gamma} e^{-\gamma R^2},
    \end{gather*}
    where $\mathcal{B}_R = \mathcal{B}_{\mathbb{R}^d} (0, R)$.
\end{lem}

\begin{proof}
    With a radial change of variables, we get
    \begin{equation*}
        M_n = C_d \int_R^{\infty} r^{n+d-1} e^{-\gamma r^2} \diff r.
    \end{equation*}
    The conclusion readily follows from Lemma \ref{lem:gauss_error_2}.
\end{proof}

\begin{proof}[Proof of Lemma \ref{lem:est_H1_gauss}]
    We recall that $\psi_j (t',x) \equiv 1$ for $x \in \mathcal{B}_j (t') \coloneqq \mathcal{B} (x_j^* (t'), \frac{v_* t}{2} + 1)$, $0 \leq \psi_j (t') \leq 1$ and $\norm{\partial_x \psi_j (t')}_{L^\infty} \leq 1$ so that the quantities of the first two estimates are all bounded by the $H^1$ norm on $\mathbb{R} \setminus \mathcal{B}_j (t')$ of $G_j (t')$ up to a multiplicative constant $C_0$ ($\norm{D^3_{xxx} \psi_k (t',x)}$ is uniformly bounded in $x \in \mathbb{R}^d$ and $t \geq 0$). Then, setting $\mathcal{B}_0 (t') \coloneqq \mathcal{B} (0, \frac{v_* t}{2} + 1)$, we easily compute
    \begin{align*}
        \int_{\mathbb{R}^d \setminus \mathcal{B}_j (t')} (\abs{G_j (t')}^2 + \abs{\nabla G_j (t')}^2 ) \diff x
            &= C_0 \int_{\mathbb{R}^d \setminus \mathcal{B}_j (t')} \Bigl( 1 + \abs{i v_j - 2 \lambda (x - x^*_j (t'))}^2 \Bigr) \exp \Bigl[ - 2 \lambda \, \abs{x - x^*_j (t')}^2 \Bigr] \diff x \\
            &\leq C_0 \int_{\mathcal{B}_0 (t')^\complement} \Bigl( 1 + \abs{y}^2 \Bigr) \exp \Bigl[ - 2 \lambda \, \abs{y}^2 \Bigr] \diff y.
    \end{align*}
    Using Lemma \ref{lem:gauss_error_3}, as soon as $\xi (t) = \frac{v_* t}{2} + 1 \geq (2 \lambda)^{-\frac{1}{2}}$, we get
    \begin{equation*}
        \int_{\mathbb{R}^d \setminus \mathcal{B}_j (t')} (\abs{G_j (t')}^2 + \abs{\nabla G_j (t')}^2 ) \diff x\leq C_0 \Bigl( \xi (t)^{d-2} + \xi (t)^{d} \Bigr) \exp \Bigl[ - 2 \lambda \, \xi (t)^2 \Bigr] = o \biggl( t^{-6} e^{- \frac{\lambda (v_* t)^2}{2}} \biggr),
    \end{equation*}
    which leads to the first two estimates of the lemma.
    The third estimate can also be deduced from a similar computation.

    As for the fourth estimate, there also holds in the same way:
    \begin{multline*}
        \int_{\mathbb{R}^d \setminus \mathcal{B}_j (t')} (\abs{x - x_j^* (t')}^4 + \abs{x - x_j^* (t')}^6) \, \abs{G_j (t')}^2 \diff x \\
        \begin{aligned}
            &\leq C_0 \int_{\mathbb{R}^d \setminus \mathcal{B}_j (t')} (\abs{x - x_j^* (t')}^4 + \abs{x - x_j^* (t')}^6) \, \exp \Bigl[ - 2 \lambda \, \abs{x - x^*_j (t')}^2 \Bigr] \diff x \\
            &\leq C_0 \int_{\mathcal{B}_0 (t')^\complement} (\abs{y}^4 + \abs{y}^6) \exp \Bigl[ - 2 \lambda \, \abs{y}^2 \Bigr] \diff y \\
            &\leq C_0 \, (\xi (t)^{d+2} + \xi (t)^{d+4}) \, \exp \Bigl[ - 2 \lambda \, \xi (t)^2 \Bigr] = o \biggl( t^{-2} e^{- \frac{\lambda (v_* t)^2}{2}} \biggr),
            \end{aligned}
    \end{multline*}
    by using again Lemma \ref{lem:gauss_error_3}.
\end{proof}

\section{Proof of Proposition \ref{prop:x_u_log_u}} \label{app:x_u_log_u}

To prove this Proposition, we use a result of \cite{Ferriere__superposition_logNLS} giving a pointwise estimate for $\abs{ G \ln \abs{G}^2 - \sum_k G_k \ln \abs{G_k}^2 }$. We recall it here in a simplified way which fits our case:

\begin{lem}[{\cite[Corollary~3.7]{Ferriere__superposition_logNLS}}] \label{cor:est_diff_g_ln_g}
    For $N \in \mathbb{N}^*$, $\lambda > 0$, $x_k \in \mathbb{R}^d$, $\omega_k \in \mathbb{R}$ and $\theta_k : \mathbb{R} \rightarrow \mathbb{R}$ a real measurable function for $k = 1, \dots, N$, and $g_k$ such that for all $x \in \mathbb{R}^d$
    \begin{equation*}
        g_k (x) = \exp \left[ i \theta_k (x) + \omega_k - \lambda \abs{x - x_j}^2 \right],
    \end{equation*}
    set
    \begin{equation*}
        g (x) = \sum_{k = 1}^N g_k (x).
    \end{equation*}
    If
    \begin{equation*}
        \varepsilon \coloneqq \left( \min_{k \neq j} \, \abs{x_{j} - x_k} \right)^{-1} < \varepsilon_0,
    \end{equation*}
    then for any $j \in \{ 1, \dots, N \}$ and for all $x \in \mathbb{R}$
    \begin{equation}
        \abs{g (x) \ln \abs{g (x)}^2 - \sum_{k = 1}^N g_k (x) \ln \abs{g_k (x)}^2} \leq 2 \sum_{k \neq j} \abs{g_k (x)} \Bigl[ \delta \omega_j + \delta \omega_k + 3 + 2 \ln{N} + \lambda \abs{x - x_k}^2 + \lambda \abs{x - x_j}^2 \Bigr], \label{est_diff_g_ln_g}
    \end{equation}
    where $\delta \omega_j \coloneqq \max_\ell \omega_\ell - \omega_j$
\end{lem}

\begin{proof}[Proof of Proposition \ref{prop:x_u_log_u}]
    Our $G_k$ satisfy the assumptions of Lemma \ref{cor:est_diff_g_ln_g}, so that \eqref{est_diff_g_ln_g} gives here for all $t\geq 0$, $j \in \{0, \dots, N \}$ and $x \in \mathbb{R}$:
    \begin{align*}
        \abs{ G (\check{t}, x) \ln \abs{G (\check{t}, x)}^2 - \sum_k G_k (\check{t}, x) \ln \abs{G_k (\check{t}, x)}^2}
            &\leq C_0 \sum_{k \neq j} \abs{G_k (\check{t}, x)} \Bigl[ 1 + \abs{x - x_k^* (\check{t})}^2 + \abs{x - x_j^* (\check{t})}^2 \Bigr] \\
            &\leq C_0 \sum_{k \neq j} \abs{G_k (\check{t}, x)} \Bigl[ 1 + t^2 + \abs{x - x_k^* (\check{t})}^2 \Bigr].
    \end{align*}
    Thus, multiplying by $\abs{x}$ and $\psi_j$ and taking the $L^2$ norm leads to:
    \begin{multline*}
        \norm{\psi_j (\check{t}) \abs{x} \abs{ G \ln \abs{G}^2 - \sum_k G_k \ln \abs{G_k}^2 }}_{L^2} \\
        \begin{aligned}
            &\leq C_0 \norm{\psi_j (\check{t}) \abs{x} \sum_{k \neq j} \abs{G_k (\check{t})} \Bigl[ 1 + t^2 + \abs{x - x_k^* (\check{t})}^2 \Bigr] }_{L^2} \\
            &\leq C_0 \norm{\psi_j (\check{t}) \Bigl(\abs{x - x_k^* (\check{t})} + \abs{x_k^* (\check{t})} \Bigr) \sum_{k \neq j} \abs{G_k (\check{t})} \Bigl[ 1 + t^2 + \abs{x - x_k^* (\check{t})}^2 \Bigr] }_{L^2} \\
            &\leq C_0 \norm{\psi_j (\check{t}) \Bigl(\abs{x - x_k^* (\check{t})} + C_0 \, t \Bigr) \sum_{k \neq j} \abs{G_k (\check{t})} \Bigl[ 1 + t^2 + \abs{x - x_k^* (\check{t})}^2 \Bigr] }_{L^2} \\
            &\leq C_0 (1 + t^3) \sum_{k \neq j} \biggl( \norm{\psi_j (\check{t}) \abs{G_k (\check{t})}}_{L^2} + \norm{\psi_j (\check{t}) \abs{x - x_k^* (\check{t})}^3 \abs{G_k (\check{t})}}_{L^2} \biggr), \\
            &\leq C_0 (1 + t^3) \sum_{k \neq j} \biggl( \norm{G_k (\check{t})}_{L^2 (\psi_j (\check{t}) \diff x)} + \norm{\abs{x - x_k^* (\check{t})}^3 \abs{G_k (\check{t})}}_{L^2 (\psi_j (\check{t}) \diff x)} \biggr).
        \end{aligned}
    \end{multline*}
    Then, using Corollary \ref{cor:est_H1_gauss} and the last estimate of Lemma \ref{lem:est_H1_gauss}, we get
    \begin{equation*}
        \norm{\psi_j (\check{t}) \abs{x} \abs{ G \ln \abs{G}^2 - \sum_k G_k \ln \abs{G_k}^2 }}_{L^2} \leq C_0 \, e^{- \frac{\lambda (v_* t)^2}{4}}.
    \end{equation*}
    Thus, we get the result by using the fact that $\sum_j \psi_j = 1$:
    \begin{align*}
        \norm{\abs{x} \abs{ G \ln \abs{G}^2 - \sum_k G_k \ln \abs{G_k}^2 }}_{L^2} &= \norm{\sum_j \psi_j (\check{t}) \abs{x} \abs{ G \ln \abs{G}^2 - \sum_k G_k \ln \abs{G_k}^2 }}_{L^2} \\
            &\leq \sum_j \norm{\psi_j (\check{t}) \abs{x} \abs{ G \ln \abs{G}^2 - \sum_k G_k \ln \abs{G_k}^2 }}_{L^2} \\
            &\leq C_0 \, e^{- \frac{\lambda (v_* t)^2}{4}}. \qedhere
    \end{align*}
\end{proof}

\bibliographystyle{abbrv}
\bibliography{sample.bib}

\end{document}